\documentclass[final,10pt,a4paper]{siamltex}
\usepackage{amsmath,cite,amstext,amsbsy,amsopn,amscd,amsxtra,upref,amssymb}
\usepackage[dvips]{graphicx}
\usepackage{color,manfnt}
\usepackage{algorithmic}
\usepackage{algorithm}
\usepackage{pictex}
\usepackage{cyrillic}

\newtheorem{example}{Example}[section]
\newtheorem{remark}{Remark}[section]
\newcommand{\R}{\mathbb R}
\newcommand{\Cplx}{\mathbb C}

\newcommand{\Cauchy}{\mathcal{C}}

\newcommand{\Hardy}{\mathcal{H}}
\newcommand{\AC}{\mathcal{A}}
\newcommand{\BV}{\mathcal{B}}
\newcommand{\FFT}{\mathbb{F}}

\newcommand{\eps}{\epsilon}
\newcommand{\pol}{\lambda}

\newcommand{\mb}[1]{\left(\begin{array}{#1}}
\newcommand{\me}{\end{array}\right)}
\newcommand{\eqb}{\begin{equation}}
\newcommand{\eqe}{\end{equation}}
\newcommand{\eqab}{\begin{eqnarray}}
\newcommand{\eqae}{\end{eqnarray}}

\newcommand{\bflambda}{{\boldsymbol\lambda}}
\newcommand{\bfpol}{{\boldsymbol\lambda}}

\newcommand{\bfphi}{{\boldsymbol\phi}}
\newcommand{\bfvarphi}{{\boldsymbol\varphi}}

\newcommand{\bfeps}{{\boldsymbol\epsilon}}

\newcommand{\imunit}{\dot{\imath\!\imath}}
\newcommand{\broje}{\mathbf{e}}

\newcommand{\bfB}{ {\mathbf B} }
\newcommand{\bfC}{ {\mathbf C} }

\newcommand{\bfE}{ {\mathbf E} }

\newcommand{\IC}{ {\mathbb C} }

\renewcommand{\theequation}{\thesection.\arabic{equation}}
\numberwithin{equation}{section}\numberwithin{equation}{section}

\title{Quadrature-Based Vector Fitting: \\ Implications For $\mathcal{H}_2$ System Approximation}

\author{Z.~Drma\v{c}\thanks{Faculty of Science, Department of Mathematics, University of Zagreb,
Bijeni\v{c}ka 30, 10000 Zagreb, Croatia.} \and S.~Gugercin\thanks{Department of Mathematics,
Virginia Polytechnic Institute and State University,
460 McBryde, Virginia Tech,
Blacksburg, VA 24061-0123.} \and C.~Beattie$^\dagger$}

\begin{document}
\maketitle

\begin{abstract}
Vector Fitting is a popular method of  constructing rational approximants designed to fit given frequency response measurements.
The original method, which we refer to as \textsf{VF}, is based on a least-squares fit to the measurements by a rational function, using an iterative reallocation of the poles of the approximant.
We show that one can improve the performance  of \textsf{VF} significantly,
by using a particular choice of frequency sampling points and properly weighting their contribution
based on quadrature rules that connect the least squares objective with an $\Hardy_2$  error measure.
Our modified approach, designated here as \textsf{QuadVF},  helps recover the original transfer function
with better global fidelity  (as measured with respect to the $\Hardy_2$ norm),
than the localized least squares approximation implicit in \textsf{VF}.
We extend the new framework also to incorporate derivative information, leading to rational approximants that minimize system error with respect to a discrete Sobolev norm. 
We  consider the convergence behavior of both \textsf{VF} and \textsf{QuadVF} as well, and evaluate potential numerical ill-conditioning of the
underlying least-squares problems. 
We  investigate  briefly \textsf{VF} in the case of noisy measurements and propose a new formulation
 for the resulting approximation problem.
Several numerical examples are provided to support the theoretical discussion.
\end{abstract}



\section{Introduction}\label{S=Intro}
In many engineering applications, the dynamics that govern phenomenae
of interest may be inaccessible to direct modeling, yet there may be an abundance of
accurate frequency response measurements available.
In such cases, one may build up an empirical dynamical system model that fits the measured frequency response data.  This empirical system may then be used as a surrogate to predict behavior or derive control strategies.

In other settings, one may have complete access to the underlying dynamical system of interest at least in principle (e.g., it may be an analytically derived computational model), however the full system may be a complex aggregate of many large subsystems, each perhaps representing diverse physics, and so it may be of such complexity that direct manipulation of the dynamical system is infeasible; potentially only simulation results would be available.   Here, one may wish to capture the dominant dynamic features of the full aggregate system and realize them with a derived dynamical system (presumably of lower order) that can replicate the response characteristics of the full aggregate system.   As before, this derived dynamical system may then be used as an efficient surrogate for the full system in contexts where performance is sensitive to model order.

A natural formulation of this task leads one to a data fitting problem using rational functions and this ultimately is our principal focus.  For convenience, we assume that the system of interest is a single-input/sin\-gle-output (SISO) linear time-invariant system associated with a transfer function, $H(s)$, that is unknown but accessible to sampling in the sense that measurements (magnitude and phase) of $H(s)$ at predetermined points, $s=\xi_1,\,\ldots,\,\xi_{\ell}$ are available.  Indeed, the values of $H(\xi_j)$, for $j=1,\ldots,\ell$ will be the only information presumed available for the system of interest.   These values may have been obtained from experimentally measured amplitude and phase responses at $\xi_j = \imunit\, 2\pi f_j$ associated with (real) driving frequencies, $f_1,\ldots, f_{\ell}$ or
 they may have been extracted via simulation from a computational model.

We derive a dynamical system (or equivalently, its transfer function)
by least squares (LS) data fitting:
 Denote by $\mathcal{R}_r$ the set of proper rational functions of order $r$
(i.e., with denominator having polynomial order $r$ and numerator having polynomial order less than $r$).
Fix $\ell$ sample points, $\{\xi_j\}_{1}^{\ell} \in \IC$, and weights, $\rho_j >0$, for $j=1,\ldots,\ell$.
The problem we address is stated succinctly as:
\begin{equation} \label{GenVFprob}
\begin{array}{c}
\mbox{Find } H_r^{\star}(s)\in \mathcal{R}_r \mbox{ such that }
\sum_{j=1}^{\ell} \rho_j\, \left| H_r^{\star}(\xi_j) - H(\xi_j) \right|^2 \longrightarrow \mbox{min} \\[.1in]
(\mbox{i.e., for all } H_r(s)\in \mathcal{R}_r,\quad
\sum_{j=1}^{\ell} \rho_j\, \left| H_r^{\star}(\xi_j) - H(\xi_j) \right|^2 \leq
\sum_{j=1}^{\ell} \rho_j\, \left| H_r(\xi_j) - H(\xi_j) \right|^2)
\end{array}
\end{equation}
Typically, all $\rho_j=1$ (the ``unweighted" case).  We will be interested in strategies that take advantage of other choices for
$\rho_j$ (which may lead to particular choices for $\xi_j$, as well).  Rational data fitting strategies brought into the service of systems identification in this way have a long history going back at least to Kalman \cite{Kalman-1958}, who computed a best least squares fit with sampled input and output data using rational functions of the form
$\sum_{j=1}^{r} a_j z^{-j}/(1+\sum_{j=1}^r b_j z^{-j})$  (in the $z$-transform domain).

Levy \cite{Levy-1959} considered (\ref{GenVFprob}), taking the rational approximants, $H_r$, to be in polynomial form:
\begin{equation} \label{LevyPolyForm}
 H_r(s)=\frac{n(s)}{d(s)}\mbox{ with }n(s)= \sum_{j=0}^{r-1}\alpha_j s^j \mbox{ and }d(s)= 1+ \sum_{j=1}^r \beta_j s^j.
 \end{equation}
Since the set of rational functions, $\mathcal{R}_r$, is not an affine set (indeed, not even convex), (\ref{GenVFprob}) is both nonlinear
and nonconvex, leading possibly to a host of local minima.
Noting first that
\begin{equation} \label{WtdLSObjective}
\sum_{j=1}^{\ell} \, \left| H_r(\xi_j) - H(\xi_j) \right|^2= \sum_{i=1}^\ell \frac{1}{ |d(\xi_i)|^2}\left| n(\xi_i) - d(\xi_i)H(\xi_i) \right|^2 ,
 \end{equation}
Levy proposed replacing (\ref{GenVFprob}) with the simpler
problem of minimizing $\sum_{i=1}^\ell |n(\xi_i) - d(\xi_i)H(\xi_i)|^2$;
an LS problem which is linear in the coefficients $\{\alpha_j\}$, $\{\beta_j\}$.
Sanathanan and Koerner \cite{Sanathanan-Koerner-1963}
argued against this tactic and provided a convincing example that such a simplification is problematic.
They  suggested an iterative approach for solving (\ref{GenVFprob}) that used Levy's simplification as a first step.

We refer to this approach as \textsf{SK} \emph{iteration} and describe two equivalent formulations of it in \S 2.   One of these formulations leads to a particularly interesting refinement, introduced by Gustavsen and Semlyen \cite{Gustavsen-Semlyen-1999} under the name \emph{Vector Fitting} (\textsf{VF}).
We describe \textsf{VF} in \S 2 and make some observations that will contribute to our analysis of it in \S 3.
Overall, \textsf{VF} has been a great success, with more than 700 citations and a wide spectrum of applications.  Many authors have applied, modified, and analyzed \textsf{VF}, see e.g. \cite{Gustavsen-2006}, \cite{Hendrickx-Dhaene-2006},
\cite{Deschrijver-Haegeman-Dhaene-2007},
\cite{Deschrijver-Gustavsen-2007}, \cite{Deschrijver-Mrozowski-Dhaene-Zutter-2008}, \cite{Deschrijver-Knockaert-Dhaene-2010}.
Our motivation for studying \textsf{VF} came initially from a desire to articulate the relationship between \textsf{VF}
and optimal rational approximation, in particular, with $\Hardy_2$-optimal model order reduction.   We set the
stage for this in \S \ref{S=VF_and_H2} where we review some basic results related to  $\Hardy_2$-optimal rational approximation.
{We show that
a small \textsf{VF} fitting error
does not necessarily correspond to small approximation error  in the $\Hardy_2$ or $\Hardy_\infty$ norm.}
This observation motivates the developments of \S \ref{sec:vf_quad_h2}, \S \ref{SS=SamplingviaQuadrature},
where we show that particular choices of sampling points and weights, as dictated by suitable quadrature formulae, may significantly
improve the performance of \textsf{VF}.
The key innovation here lies in reformulating the problem essentially
as an approximation problem in a normed function space instead of as an algebraic LS problem.

Some implementation details are provided in \S \ref{sec:prac_issues}. Formal mathematical justification of
mirroring unstable nodes in \textsf{VF} is given in \S \ref{SSS::Mirroring-and-scaling}.
In \S \ref{S=Convergence}, we  use numerical examples to illustrate
the complexity  of the theoretically open problem of the convergence of \textsf{VF} iterations.
In \S \ref{SS=CaveatConditioning}, we discuss the important issue of high condition numbers of the matrices used
in VF, and introduce a regularized LS version of \textsf{VF}.
The behavior of \textsf{VF} in the case of noisy data is analyzed in \S \ref{SS=noisy},
where we show that \textsf{VF} will asymptotically and implicitly solve a structured total least squares problem in computing the
coefficients.   This goes some distance in explaining the robustness observed in \textsf{VF}.

In recent years, the Loewner framework, initially introduced by Mayo and Antoulas \cite{Mayo-2007} and further extended in \cite{AIL11,IonitaAntoulas2013,lefteriu2010new}, has emerged as a  powerful,  effective and numerically efficient method to construct rational approximants directly from frequency domain measurements.  Our major focus in this paper is the rational least-squares approximation produced by \textsf{VF}; to investigate \textsf{VF} from an optimal approximation perspective, to offer improvement based on this analysis and to examine several computational issues.  A comparison of \textsf{VF} with the Loewner framework and related approaches is natural to consider however it will not be considered here.

\section{The Sanathanan-Koerner  Iteration and Vector Fitting}
\subsection{SK iteration}
 Sanathanan and Koerner \cite{Sanathanan-Koerner-1963} noted
 that minimizing the objective function $\sum_{i=1}^\ell |n(\xi_i) - d(\xi_i)H(\xi_i)|^2$ instead of (\ref{WtdLSObjective})
 could produce quite different outcomes since
$|d(\xi_i)|$ could vary over a wide range of magnitudes.  They offered an alternative approach through the iterative adjustment of
the LS weights:
\begin{equation}\label{eq:relaxedNLS}
\begin{array}{c}
\mbox{Starting with  }d^{(0)}(s)\equiv 1,\mbox{ solve successively for }k=0,1,2,\ldots \\
\displaystyle \sum_{i=1}^\ell  \left| \frac{n^{(k+1)}(\xi_i) - d^{(k+1)}(\xi_i)H(\xi_i)}{d^{(k)}(\xi_i)} \right|^2 \longrightarrow \min.
\end{array}
\end{equation}
We will refer to this process as ``\textsf{SK} iteration."

\paragraph{Polynomial Representation} Using the polynomial representation of $H_r(s)$ in (\ref{LevyPolyForm}),
one may reformulate  \eqref{eq:relaxedNLS} as a weighted LS problem (following \cite{Sanathanan-Koerner-1963}):
\begin{equation} \label{eq:B=Vandermonde}
 \|\Delta^{(k)} ( \BV y^{(k+1)}-h )\|_2\rightarrow\min \
\end{equation}
where the optimization parameters are
$y^{(k+1)} = \left(\begin{smallmatrix} \alpha_0^{(k+1)} & \alpha_1^{(k+1)} & \cdots & \alpha_{r-1}^{(k+1)} & \beta_1^{(k+1)} & \beta_2^{(k+1)} & \cdots & \beta_r^{(k+1)}\end{smallmatrix}\right)^T,$
while
\begin{equation} \label{eq:Bmatrix}
\begin{array}{c}
\BV = \left(\begin{smallmatrix}
1      & \xi_1  & \xi_1^2 & \ldots & \xi_1^{r-1} & -H(\xi_1)\xi_1 & -H(\xi_1)\xi_1^2 & \ldots & -H(\xi_1) \xi_1^r \cr
1      & \xi_2  & \xi_2^2 & \ldots & \xi_2^{r-1} & -H(\xi_2)\xi_2 & -H(\xi_2)\xi_2^2 & \ldots & -H(\xi_2) \xi_2^r \cr
\vdots & \vdots & \vdots  & \vdots & \vdots      & \vdots         & \vdots           & \vdots & \vdots \cr
1 & \xi_{\ell-1} & \xi_{\ell-1}^2 & \ldots & \xi_{\ell-1}^{r-1} & -H(\xi_{\ell-1})\xi_{\ell-1} & -H(\xi_{\ell-1})\xi_{\ell-1}^2
& \ldots & -H(\xi_{\ell-1}) \xi_{\ell-1}^r \cr
1 & \xi_\ell & \xi_\ell^2 & \ldots & \xi_\ell^{r-1} & -H(\xi_\ell)\xi_\ell & -H(\xi_\ell)\xi_\ell^2 & \ldots & -H(\xi_\ell) \xi_\ell^r \cr
\end{smallmatrix}\right),\quad
h = \left(\begin{smallmatrix} H(\xi_1) \cr H(\xi_2) \cr \vdots \cr H(\xi_{\ell-1}) \cr  H(\xi_\ell)\end{smallmatrix}\right), \\[.3in]
\qquad \mbox{and }\qquad \Delta^{(k)} = \mathrm{diag}\left(\frac{1}{|d^{(k)}(\xi_j)|}\right)_{j=1}^\ell.
\end{array}
\end{equation}

The sequence of LS solutions, $y^{(k)}$, yields polynomial coefficients for the sequence of numerators, $n^{(k)}(s)$, and denominators, $d^{(k)}(s)$, of $H_r^{(k)}(s)$ (as in (\ref{LevyPolyForm})).  If the denominator sequence, $d^{(k)}(s)$, converges,  then so does the numerator sequence, $n^{(k)}(s)$, and so the \textsf{SK} iteration (\ref{eq:relaxedNLS}) produces a system $H_r(s)$ that may be expected to be a locally optimal solution to (\ref{GenVFprob}).

\paragraph{Barycentric representation} The rational function $H_r(s)$ in  (\ref{LevyPolyForm}) can be
represented alternatively in barycentric form, which happens here to be both elegant and useful.
We develop this by expressing the numerator and the denominator in a Lagrange interpolating basis:
Pick an arbitrary set of mutually distinct scalars $\lambda_0, \lambda_1,\ldots, \lambda_r$ (``interpolation points")
and define the nodal polynomial
$\omega_r(s)=\prod_{k=1}^r (s-\lambda_k)$ (notice $\lambda_0$ is excluded).
Then,
$$
n(s) =\omega_r(s) \sum_{j=1}^r  \frac{w_j\, n(\lambda_j) }{s-\lambda_j}
\quad\mbox{ and }\quad
d(s) = \omega_r(s)\left(\alpha+ \sum_{j=1}^r  \frac{w_j\, d(\lambda_j) }{s-\lambda_j}\right) ,
$$
where $w_j = 1 / \prod_{k\neq j}(\lambda_j-\lambda_k)$ enforces interpolation of $n(s)$, and hence $H_r(s)$, at $s=\lambda_j$ for $j=1,\ldots,r$ and choosing $\alpha=\frac{d(\lambda_0)}{\omega_r(\lambda_0)}-\sum_{j=1}^r  \frac{d(\lambda_j) w_j}{\lambda_0-\lambda_j}$
then enforces interpolation of $H_r$  also at $s=\lambda_0$.  As long as $d(s)$ has polynomial degree $r$, then $\alpha\neq 0$.
Define $\phi_j=\frac{w_j}{\alpha}n(\lambda_j)$ and $\varphi_j=\frac{ w_j}{\alpha}d(\lambda_j)$, so
\begin{equation}\label{eq::barycentric-form-def}
H_r(s) = \frac{{\sum_{j=1}^r \frac{\phi_j}{s - \pol_j}}}
{{1+ \sum_{j=1}^r \frac{\varphi_j}{s - \pol_j}}} = \frac{\tilde{n}(s)}{\tilde{d}(s)}
\qquad \mbox{with }\left\{ \begin{array}{l} \tilde{n}(s)= \sum_{j=1}^r \frac{\phi_j}{s - \pol_j},\mbox{ and}\\
\tilde{d}(s) =1+ \sum_{j=1}^r \frac{\varphi_j}{s - \pol_j} \; .\end{array}\right.
\end{equation}
We may now use $\phi_j, \varphi_j$  as optimization parameters in each step of the \textsf{SK} iteration (\ref{eq:relaxedNLS}).
Indeed,  for a given set of interpolation points,  $\lambda_1,\ldots, \lambda_r$, observation points
 $ \xi_1,\ldots, \xi_{\ell}$,
and system observations $ H(\xi_1),\ldots, H(\xi_{\ell})$, the  parameters $\phi_j^{(k)}, \varphi_j^{(k)}$ describe
$H_r^{(k)}(s)=\frac{\tilde{n}^{(k)}(s)}{\tilde{d}^{(k)}(s)}$ in the $k$th step of (\ref{eq:relaxedNLS}), replacing
$n^{(k)}$ and $d^{(k)}$ in (\ref{eq:relaxedNLS}) with
\begin{equation}\label{new_n_d}
\tilde{n}^{(k)}(s)= \sum_{j=1}^r \frac{\phi_j^{(k)}}{s - \pol_j}
\quad \mbox{ and }\quad
\tilde{d}^{(k)}(s) = 1+ \sum_{j=1}^r \frac{\varphi_j^{(k)}}{s - \pol_j},
\end{equation}
respectively.
Now, $\phi_j^{(k)}, \varphi_j^{(k)}$ are determined by  solution of the  successive least squares problems
\begin{equation}\label{eq:LS}
\|\Delta^{(k)} ( \AC x^{(k+1)}-h )\|_2\rightarrow\min,\;\; k = 0, 1, 2, \ldots,
\end{equation}
where the unknowns now are
$x^{(k+1)} = \left(\begin{smallmatrix} \phi_1^{(k+1)} & \phi_2^{(k+1)} & \cdots & \phi_r^{(k+1)} & \varphi_1^{(k+1)} & \varphi_2^{(k+1)} & \cdots & \varphi_r^{(k+1)}\end{smallmatrix}\right)^T $ and
\begin{equation} \label{eq:A=PC}
\AC = \left(\begin{smallmatrix}
\frac{1}{\xi_1-\pol_1} & \frac{1}{\xi_1-\pol_2} & \cdots & \frac{1}{\xi_1-\pol_r} & \frac{-H(\xi_1)}{\xi_1-\pol_1} &
\frac{-H(\xi_1)}{\xi_1-\pol_2} & \cdots & \frac{-H(\xi_1)}{\xi_1-\pol_r} \\[0.3em]
\frac{1}{\xi_2-\pol_1} & \frac{1}{\xi_2-\pol_2} & \cdots & \frac{1}{\xi_2-\pol_r} & \frac{-H(\xi_2)}{\xi_2-\pol_1} &
\frac{-H(\xi_2)}{\xi_2-\pol_2} & \cdots & \frac{-H(\xi_2)}{\xi_2-\pol_r} \cr
\vdots & \vdots & \vdots & \vdots & \vdots & \vdots & \vdots & \vdots \cr
\frac{1}{\xi_{\ell-1}-\pol_1} & \frac{1}{\xi_{\ell-1}-\pol_2} & \cdots & \frac{1}{\xi_{\ell-1}-\pol_r} & \frac{-H(\xi_{\ell-1})}{\xi_{\ell-1}-\pol_1} &
\frac{-H(\xi_{\ell-1})}{\xi_{\ell-1}-\pol_2} & \cdots & \frac{-H(\xi_{\ell-1})}{\xi_{\ell-1}-\pol_r}\cr
\frac{1}{\xi_\ell-\pol_1} & \frac{1}{\xi_\ell-\pol_2} & \cdots & \frac{1}{\xi_\ell-\pol_r} & \frac{-H(\xi_\ell)}{\xi_\ell-\pol_1} &
\frac{-H(\xi_\ell)}{\xi_\ell-\pol_2} & \cdots & \frac{-H(\xi_\ell)}{\xi_\ell-\pol_r}
\end{smallmatrix}\right).
\end{equation}
Note $h$ and $\Delta^{(k)}$ are as defined in (\ref{eq:Bmatrix}), with $\tilde{d}^{(k)}(s)$ as
given in (\ref{new_n_d})  replacing  $d^{(k)}(s)$ in $\Delta^{(k)}$.

\paragraph{Equivalence of the representations}
It is straightforward to see that both (\ref{eq:LS})-(\ref{eq:A=PC}) and (\ref{eq:B=Vandermonde})-(\ref{eq:Bmatrix}) are
simply different representations of the same iteration step described in (\ref{eq:relaxedNLS}),
the key difference being that $H_r(s)$ is expressed with respect to different bases.
Note that $\BV$ in (\ref{eq:B=Vandermonde})-(\ref{eq:Bmatrix}) depends solely on the
complex frequency points,  $\xi_i$, at which the system is observed, while
$\AC$ in (\ref{eq:LS})-(\ref{eq:A=PC}) depends both on  those observation points, $\{\xi_i\}$ and on auxiliary interpolation points,
$\{\lambda_j\}$.  The interpolation points ($\lambda$s) used in the definition of $\AC$ have been chosen arbitrarily; they serve
just to fix a particular barycentric representation, and remain constant throughout the iteration.

Interestingly, if the interpolation points used in the definition of $\AC$ are chosen to be the $r$th roots of unity,
$\lambda_j=\overline{\omega}^{j-1}$ with $\omega = \broje^{\imunit(2\pi/r)}$, then one can show that the \textsf{SK} iterations
in \eqref{eq:B=Vandermonde} and \eqref{eq:LS}
are related via the $r$-dimensional discrete Fourier Transform, $\FFT\in\IC^{r\times r}$ with $\FFT_{ij}=\frac{\omega^{(i-1)(j-1)}}{\sqrt{r}}$.
More precisely,
solving $\|\Delta^{(k)} ( \BV y^{(k+1)}-h )\|_2\rightarrow\min$  in the usual polynomial basis is equivalent to solving
$\| \tilde{\Delta}^{(k)} ( \AC \tilde{x}^{(k+1)}  -  D_1^{-1}h )\|_2 \rightarrow\min$ with a particular choice of barycentric representation,
and the two solutions are related by
\begin{equation}
y^{(k+1)}=\FFT D_2 \tilde{x}^{(k+1)},\;\;\mbox{where}\;\; (D_1)_{ii} =\frac{\xi_i^n-1}{\sqrt{r}},~~
(D_2)_{jj} = \frac{1}{\omega^{j-1}}.
\end{equation}

Each of the iterative processes described in (\ref{eq:B=Vandermonde})-(\ref{eq:Bmatrix})
and in (\ref{eq:LS})-(\ref{eq:A=PC}) are concrete realizations of (\ref{eq:relaxedNLS}), and as such,
they each are driven by successive updates of the weighting factors  $\Delta^{(k)}$.
As the weighting factors, $\Delta^{(k)}$, change, so too do the denominators of
the approximants $H_r^{(k)}(s) = \frac{n^{(k)}(s)}{d^{(k)}(s)}$ and, in particular, the
poles of  $H_r^{(k)}(s)$ will change from step to step.
No constraint has been imposed that guarantees these poles remain in the left half-plane, and so it may happen that a minimizing solution to (\ref{eq:relaxedNLS}) produces an unstable system, an outcome that would generally be viewed as unsatisfactory.
Thus, as a practical matter, it is necessary additionally to monitor the zeros of the denominators, $d^{(k)}(s)$, and, perhaps on occasion,
intercede to repair unstable poles as they emerge (e.g., by reflecting them across the imaginary axis back into the left half-plane).
\emph{Vector Fitting}, as we see next, also uses this information to determine an advantageous representation for the next step in (\ref{eq:relaxedNLS}).

\subsection{Vector Fitting (\textsf{VF}) \cite{Gustavsen-Semlyen-1999}}  \label{subsec:VF}
Since the choice of the interpolation points in the \textsf{SK} iteration only determines a particular
barycentric representation for rational functions, one is free to change the $\lambda_j$ at every step.
The original formulation of \emph{Vector Fitting} as introduced by Gustavsen and Semlyen \cite{Gustavsen-Semlyen-1999}
takes advantage of this flexibility
and cleverly updates the interpolation points in the course of the iteration. In addition to
{providing more accurate rational approximants} and generally providing greater stability and better performance than
 the \textsf{SK} iteration,
 this dynamic updating of the interpolation points achieves other  useful goals, as explained below and in \S\ref{SS=noisy}.

Suppose now that the interpolation points depend on $k$ and denote them by $\lambda_j^{(k)}$;
we define $\AC^{(k)}\equiv \AC(\bflambda^{(k)})$ to be $\AC$  as defined in (\ref{eq:A=PC}), but with $\lambda_j$ replaced by $\lambda_j^{(k)}$.
After the $k$-th step of the iteration,
\textsf{VF} assigns $\lambda_j^{(k+1)}$ to be the zeros
of $\tilde{d}^{(k)}(s)$ in (\ref{new_n_d}):
\begin{equation} \label{PoleZero_d_k}
\tilde{d}^{(k)}(s) = 1+  \sum_{j=1}^r \frac{\varphi_j^{(k)}}{s-\lambda_j^{(k)}}
= \frac{\prod_{j=1}^r (s-\lambda_j^{(k+1)})}{\prod_{j=1}^r (s-\lambda_j^{(k)})}.
\end{equation}
 Then, the goal of \eqref{eq:LS} becomes the minimization of
\begin{align}
 \|\Delta^{(k)}(\AC^{(k)} x^{(k+1)} - h )\|_2^2 &= \sum_{i=1}^\ell \frac{1}{|\tilde{d}^{(k)}(\xi_i)|^2} \left| \sum_{j=1}^r \frac{\phi_j^{(k+1)}}{\xi_i-\lambda_j^{(k)}}
- H(\xi_i) \left(1+\sum_{j=1}^r \frac{\varphi_j^{(k+1)}}{\xi_i-\lambda_j^{(k)}}\right) \right|^2  \\
 &=\sum_{i=1}^\ell \left| \frac{\prod_{j=1}^r (\xi_i-\lambda_j^{(k)})}{\prod_{j=1}^r (\xi_i-\lambda_j^{(k+1)})}\right|^2
\left| \frac{\tilde{p}^{(k+1)}(\xi_i)}{\prod_{j=1}^r (\xi_i-\lambda_j^{(k)})} - H(\xi_i) \frac{\tilde{q}^{(k+1)}(\xi_i)}{\prod_{j=1}^r (\xi_i-\lambda_j^{(k)})}\right|^2  \nonumber
\end{align}
where $\tilde{p}^{(k+1)}$ and $\tilde{q}^{(k+1)}$ are, respectively, polynomials of degree $r-1$ and $r$.
 Continuing with similar algebraic manipulations, one obtains
\begin{eqnarray}
 \|\Delta^{(k)}(\AC^{(k)} x^{(k+1)} - h )\|_2^2& = & \sum_{i=1}^\ell
\left| \frac{\tilde{p}^{(k+1)}(\xi_i)}{\prod_{j=1}^r (\xi_i-\lambda_j^{(k+1)})} - H(\xi_i) \frac{\tilde{q}^{(k+1)}(\xi_i)}{\prod_{j=1}^r (\xi_i-\lambda_j^{(k+1)})}\right|^2  \nonumber\\
&=& \sum_{i=1}^\ell \left| \sum_{j=1}^r \frac{\tilde\phi_j^{(k+1)}}{\xi_i-\lambda_j^{(k+1)}}
- H(\xi_i) \left( 1 + \sum_{j=1}^r \frac{\tilde\varphi_j^{(k+1)}}{\xi_i-\lambda_j^{(k+1)}} \right) \right|^2 \label{tildephi} \\
&=& \|\AC^{(k+1)} \tilde{x}^{(k+1)} - h \|_2^2 , \nonumber
\label{eq:VF_unweighted}
\end{eqnarray}
where $\tilde{x}^{(k+1)} = \left( \tilde{\phi}_1^{(k+1)} ~ \tilde{\phi}_2^{(k+1)} ~ \cdots ~ \tilde{\phi}_r^{(k+1)} ~ \tilde{\varphi}_1^{(k+1)} ~ \tilde{\varphi}_2^{(k+1)} ~ \cdots ~ \tilde{\varphi}_r^{(k+1)}\right)^T$
with $\tilde{\phi}_j^{(k+1)}$ and   $\tilde{\varphi}_j^{(k+1)}$ as defined in \eqref{tildephi}.
Thus,
one step of \textsf{VF} corresponds  to solving the least squares problem
\begin{equation}\label{eq:VF}
\|  \AC^{(k+1)} \tilde{x}^{(k+1)}-h \|_2\rightarrow\min,\;\; k = 0, 1, 2, \ldots
\end{equation}
This is an unweighted LS step using an updated barycentric
representation of $H_r(s)$ based on $\bflambda^{(k+1)}$ and with the coefficient matrix $\AC^{(k+1)}=\AC(\bflambda^{(k+1)})$;
effectively, one step of the \textsf{SK} iteration with unity weighting. For this reason,  \textsf{VF} may be thought of as
a representation of \textsf{SK} iteration in a well-chosen basis \cite{Hendrickx-Dhaene-2006}.
One of the points we make in this paper is that \textsf{VF}
is more than that.

The scaling that underlies the \textsf{SK} iteration is implicit in  (\ref{eq:VF_unweighted}) and
provides a critical correction to the approximation metric when
close to the true minimizer.  However, when the approximant, $H_r^{(k)}$ is  far from the true minimizer,
that same scaling  may inflict severe damage on the early evolution of the iterations,
leading subsequent  iterates to an unsatisfactory final approximant (cf. \cite{Sanathanan-Koerner-1963}).
This makes the performance of the \textsf{SK} iteration (and hence also the \textsf{VF} iteration) potentially sensitive to the quality of initialization.

Since \textsf{VF} assigns $\lambda_j^{(k+1)}$ to be the zeros
of $\tilde{d}^{(k)}(s)$, the poles of $\tilde{d}^{(k+1)}(s)$ will be zeros of $\tilde{d}^{(k)}(s)$ and in the limit,
assuming convergence, pole-zero cancelation occurs.  If the interpolation points, $\lambda_j^{(k)}$, converge to
finite values as $k\rightarrow \infty$ then  from (\ref{PoleZero_d_k}),
$\tilde{d}^{(k)}(s) \rightarrow 1$ and, in the limit, $\tilde{n}^{(k)}(s)$ will give the final rational approximant in the pole-residue
representation. However, theoretical convergence of  \textsf{VF} is still an open problem, and
a careful justification
of the stopping criterion (e.g. using backward error analysis) is also lacking.
We address these issues in more detail in \S \ref{S=Convergence}.

\section{Vector Fitting and Discrete $\Hardy_2$ Approximation}\label{S=VF_and_H2}

\subsection{$\Hardy_2$ approximation}  \label{sec:h2approx}
Let $\mathcal{H}_2(\Cplx_+)$ denote the vector space  of complex functions, $H(s)$,
that are analytic in the open right-half plane,
$\Cplx_+ = \{s\equiv x+\imunit y \in \Cplx\; :\; x >0\}$,
such that
$
\sup_{x>0}\int_{-\infty}^{+\infty}|H(x+\imunit y)|^2 dy < \infty .
$
$\mathcal{H}_2(\Cplx_+)$ is a Hilbert space endowed  with an inner product
\begin{equation} \label{eq:h2inner}
\left< G, H\right>_{\Hardy_2} = \frac{1}{2\pi}\int_{-\infty}^{+\infty}  \overline{H(\imunit\omega)}\, G(\imunit\omega)\, d\omega,\;\;
\mbox{and norm}\;\; \|G\|_{\Hardy_2}=\sqrt{\left<G,G\right>_{\Hardy_2}}\,.
\end{equation}
The boundary operator isometry $\mathcal{T} : \mathcal{H}_2(\Cplx_+) \longrightarrow L_2(\imunit\R)$,
$\mathcal{T}[H](\imunit\omega)=\lim_{x\downarrow 0} H(x+\imunit\omega)$, identifies $H$ with its boundary
function, $\mathcal{H}_2(\Cplx_+)\cong \mathrm{Range}(\mathcal{T}) \subset L_2(\imunit\R)$.
If $G$ and $H$  are strictly proper rational functions representing transfer functions of real stable linear time invariant
dynamical systems then $G,H\in \mathcal{H}_2(\Cplx_+)$, and we have in addition,
$$
\left< G, H\right>_{\Hardy_2} = \left< H, G\right>_{\Hardy_2}=
 \frac{1}{2\pi}\int_{-\infty}^{+\infty} H(-\imunit\omega)\, G(\imunit\omega)\, d\omega \quad
 \mbox{and} \quad
G(s) = \frac{1}{2\pi}\int_\R \frac{G(\imunit\omega)}{s-\imunit\omega} d\omega.
$$
If $H_r$ is an $\Hardy_2$-optimal $r$th order rational approximation to a given $H(s)\in \Hardy_2$, then it must be a Hermite interpolant of $H(s)$ in the following sense:  Suppose
$$
H_r(s) = \sum_{i=1}^r \frac{\phi_i}{s-\lambda_i} = \mathrm{arg}\hspace{-0.8em}\min_{\hspace{-1em}\mbox{\tiny order }\widetilde{H}_r\leq r \atop \hspace{-0.5em}\widetilde{H}_r\, \mbox{\tiny  stable} }\| H - \widetilde{H}_r\|_{\Hardy_2}.
$$
Then,
\begin{equation} \label{eq:h2optcond}
H(-\lambda_j) = H_r(-\lambda_j)\qquad{\rm and}\qquad H'(-\lambda_j) = H_r'(-\lambda_j),\qquad{\rm for}~~~j=1,2,\ldots,r;
\end{equation}
$H_r(s)$ is a Hermite interpolant to $H(s)$ at the mirror images of its own poles reflected across the imaginary axis \cite{MeierLuenberger:1967:Approximation,Gugercin_Antoulas_Beattie:2008:IRKA}.
 These optimal interpolation points, $\{-\lambda_i\}_{i=1}^r,$ evidently depend on the poles of the optimal approximant that is sought, so they are not known \emph{a priori}.  The \emph{Iterative Rational Krylov Algorithm} (\textsf{IRKA}) of Gugercin et al. \cite{Gugercin_Antoulas_Beattie:2008:IRKA} is a numerically effective iterative correction process that systematically enforces these necessary conditions for optimality.

The original formulation of \textsf{IRKA} described in  \cite{Gugercin_Antoulas_Beattie:2008:IRKA} requires access to a first-order state-space realization for $H(s)$:
$H(s) = \bfC (s \,\bfE - \mathbf{F})^{-1}\bfB$.   By employing a Loewner-matrix framework introduced by Mayo and Antoulas
\cite{Mayo-2007}, Beattie and Gugercin \cite{Beattie_Gugercin::RealizationIndependent}  relaxed this requirement; one only needs the ability to evaluate $H(s)$ for $s \in \IC$ in order to obtain (locally) $\Hardy_2$-optimal rational approximants to $H(s)$.
This has allowed effective data-driven $\Hardy_2$-optimal system approximation for a much larger class of functions, including many that are not necessarily rational such as arise with delay systems.
 For more details on optimal $\Hardy_2$ approximation, see \cite{Gugercin_Antoulas_Beattie:2008:IRKA,Ant2010imr,MeierLuenberger:1967:Approximation,wilson1970optimum,spanos1992anewalgorithm} and references therein.

 Notably, the data required to run the  Loewner-\textsf{IRKA} approach of \cite{Beattie_Gugercin::RealizationIndependent} is similar to what is required for \textsf{VF} but with one important difference,  neither the number nor the location of the points of evaluation of $H(s)$ is known in advance for  the Loewner-\textsf{IRKA} approach.  This is in contrast to \textsf{VF} where a predetermined number of $H(s)$ evaluations are computed (or provided by simulation) at the beginning and the rest of the process does not require any new $H(s)$ evaluations. This, of course, comes with the disadvantage that the resulting approximation due to
 \textsf{VF} will fit only the sampling of $H(s)$ that had been acquired and so it ultimately may be a poor approximation to $H(s)$ with respect to an $\mathcal{H}_2$ or $\Hardy_\infty$ measure.

\subsection{Reformulating Vector Fitting as Discrete $\mathcal{H}_2$ minimization} \label{sec:vf_quad_h2}
\textsf{VF} is widely recognized as a very effective tool in creating rational approximants that fit frequency-sampled functions.  How best to organize the necessary frequency sampling is not discussed in general and seems governed more by expedience with just a few general guidelines.  For example, in the discussion portion of \cite{Gustavsen-Semlyen-1999}, the authors offer the heuristic "The samples should be chosen so densely that the frequency response is fully resolved. "  They go on to recommend having at least as many samples as poles ($r$) and, in turn, at least twice as many poles as there are peaks in the frequency response.  These are useful guidelines, yet clearly they do not (nor are they intended to) cover all cases of interest: for example, high modal densities can obscure resonances.  Moreover, if significant expense is associated with obtaining each frequency sample, then one is motivated to reduce sampling density and one may be forced to enter the gray area between a sampling density that ``fully resolves" the frequency response and one that may leave it ``unresolved."   Indeed, certain application settings may not allow sufficient sampling density to resolve fully the frequency response and one wishes then to maximize the effectiveness of parsimonious sampling strategies.
 \begin{example}
 {\em
Consider the FOM Model from the \textsf{NICONET} Benchmark collection \cite{NICONET-report}. The model $H(s)$ has order
 $n=1006$, yet the frequency response has only three obvious peaks, between 8 Hz and 160 Hz.
 We create a rational approximant of order $r=12$ using \textsf{VF} with $\ell=40$ frequency sampling points $f_i$, logarithmically spaced between $10^{-3}$ and $10^3$.
  \textsf{VF} was very effective in producing a
rational approximant with an excellent goodness-of-fit; the relative least-squares residual
was $ 2.75 \times 10^{-4}$. However, this did not mean a high-fidelity model was obtained:  indeed, the corresponding relative $\Hardy_2$ error was  only $1.78 \times 10^{-1}$ and much better models of the same order can be obtained easily.
Applying \textsf{IRKA} to the same system produced a model of the same order, but with a relative $\Hardy_2$ error of only $1.92 \times 10^{-4}$, an approximation that is virtually indistinguishable from the original.   Not surprisingly, this greater accuracy came at a somewhat greater cost:
On this example, \textsf{IRKA} took $5$ iterations to converge. Every iteration step required twelve $H(s)$ evaluations and twelve $H'(s)$ evaluations. However, the twelve interpolation points comprised $3$ complex conjugate pairs and $6$ real points in each iteration, so every iteration required only  nine independent $H(s)$ and nine independent $H'(s)$ evaluations, netting a total of  $\ell = 45$ $H(s)$ and $\ell = 45$ $H'(s)$ evaluations.  The main point to note in this regard is not so much the number of function/derivative evaluations --- it is often the case that function and derivative evaluations can be combined so the net computational effort, both in this case and in general, is typically far less than twice what is required just for function evaluations. Rather, one should note that with \textsf{IRKA} (and in contrast with \textsf{VF}), one cannot anticipate exactly \emph{where} these function evaluations will occur beforehand.
}
 \end{example}

Our goal is to bring the achievable accuracy of \textsf{VF} more in line with what \textsf{IRKA} can provide, without sacrificing its attractive computational features.   We find that by interpreting the \textsf{VF} objective function of (\ref{GenVFprob}) as a discretization of an $\Hardy_2$ error measure, remarkably effective sampling strategies may be developed systematically through numerical quadrature.   The general approach that we will take in the sequel arrives at a vector fitting formulation (\ref{GenVFprob})  by approximating the $\Hardy_2$ error with an appropriate quadrature rule.
This will lead us to minor modifications of \textsf{VF} that we find often dramatically improves its quality of approximation.

\subsection{Effective Sampling Points via Quadrature}\label{SS=SamplingviaQuadrature}

Approximating the $\Hardy_2$ error measure with
  a quadrature rule leads one to consider
 approximations of the form
{\small
 \begin{eqnarray}\label{eq:H2error_discretized}
\int_{-\infty}^{+\infty}|H(\imunit\omega)-H_r(\imunit\omega)|^2 d\omega &\approx&
\sum_{j=1}^\ell \rho_j^2  |H( \xi_j) - H_r( \xi_j)|^2 + \rho_{+}^2\,M_{+}[|H-H_r|^2] +\rho_{-}^2\,M_{-}[|H-H_r|^2]
\end{eqnarray}
}
 where $M_{\pm}[G]$ are linear functionals of $G$ that capture information about behavior at $\pm\infty$.
Note that if $\rho_{+}=\rho_{-}=0$, with all other $\rho_j=1$,
and if sampling nodes, $\xi_j$,  are chosen to be equidistant and in complex conjugate pairs, then we
recover the usual \textsf{VF} objective function which then can be understood
as a composite trapezoid quadrature rule for the integral in (\ref{eq:H2error_discretized}),
giving the $\Hardy_2$ error.

Of course, the trapezoid rule will not be an optimal choice of quadrature rule in most cases and
many, much more effective options are easily formulated, many of which involve first mapping
the unbounded domain of integration, $(-\infty,\infty)$, to a finite interval, often either $(-1,1)$ or $(0,\pi)$,
and then applying a high accuracy quadrature rule.   We  focus on a quadrature rule developed by Boyd \cite{Boyd-1987}, which is related to Clenshaw-Curtis quadrature and chosen here for its simplicity.  Many options of this sort may be considered;  our main goal is to illustrate the potential of this approach without overburdening the reader with technicalities.

Adapted to our setting, the Boyd/Clenshaw-Curtis (\textsf{B/CC}) formula \cite{Boyd-1987} is
\begin{align}\label{eq:Boyd:formula}
\|H(s)\|_{\Hardy_2}^2=\int_{-\infty}^{+\infty}|H(\imunit\omega)|^2 \,d\omega
&  =  \int_0^\pi \frac{L}{\sin^2 t} |H(\imunit L \cot t)|^2 \, dt \nonumber \\
 \approx \sum_{j=1}^\ell  \frac{L\pi}{(\ell+1)\sin^2 t_j} & |H(\imunit L \cot t_j)|^2 +
\frac{\pi}{2L(\ell+1)}\left(|M_{+}[H]|^2+|M_{-}[H]|^2\right).
\end{align}
where
$L>0$ is  a freely chosen scaling parameter, $t_j = \frac{j\pi}{\ell+1}, $ for $j=1,\ldots, \ell$,  and
\begin{equation} \label{first_last_terms}
\begin{array}{c}
\displaystyle M_{+}[G] = \lim_{\omega\rightarrow \infty} \imunit\omega\, G(\imunit \omega) =\lim_{t\rightarrow 0^{+}}\frac{G(\imunit L \cot t)}{\sin(t)}\cdot \,\imunit L \\[.2in]
\displaystyle M_{-}[G] = \lim_{\omega\rightarrow -\infty} \imunit\omega\, G(\imunit \omega) =\lim_{t\rightarrow \pi^{-}}\frac{G(\imunit L \cot t)}{\sin(t)}\cdot \, \imunit L
\end{array}
\end{equation}
For example, if  $H(s)$ is a strictly proper transfer function with realization,  $H(s) = \bfC(sI-\mathbf{F})^{-1}\bfB$, then
$M_{+}[H]=M_{-}[H]=\bfC\bfB$.

The choice of $L$ can  influence greatly the accuracy of this quadrature rule.  Notice that as the value of $L$ decreases, the quadrature nodes are drawn towards the origin with diminished weight, while contributions at $\pm\infty$ have increased weight to compensate.    Boyd \cite{boyd1982optimization} observed that when integrands are entire functions,  accuracy may be increased optimally by increasing $L$ in a way that is dependent on the order of the quadrature rule ($\ell$) and the growth of the integrand at $\infty$.
However, if the integrand is meromorphic,  increasing $L$ will also draw singularities toward the sampling domain, and accuracy will eventually degrade.   Choosing $L$ optimally to balance these two effects is nontrivial, and Boyd \cite{boyd1982optimization} offers concrete strategies and an insightful discussion.
To illustrate the effect of different choices for $L$, we used
(\ref{eq:Boyd:formula}) to compute the ${\mathcal{H}_2}$ norm of the Heat Model from the \textsf{NICONET} Benchmark collection \cite{NICONET-report}.  With only $20$ function evaluations and using $L=0.486$, we approximated  $\|H\|_{\mathcal{H}_2}$ with a relative error of $2.8\cdot 10^{-7}$.  Even using only $10$ function evaluations (while keeping the same $L$ value) resulted in a relative error of $2.2\cdot 10^{-4}$.    When one considers that the usual computational task involved in computing the ${\mathcal{H}_2}$ norm involves the solution of a (large) Lyapunov equation, the ability to compute the ${\mathcal{H}_2}$ norm to such great accuracy with only $10$ function evaluations suggests the power that effective numerical quadrature can bring.   Note that in this example, the function behaves quite well.  If the function has many nearly unstable poles, then determining an optimal $L$ will not be as simple.   To provide some contrast, if we decrease $L$ to $L=0.1$ then with $20$ function evaluations, the ${\mathcal{H}_2}$  norm is estimated with a worse relative error of $2.7\cdot 10^{-4}$. Likewise, if we increase $L$ to $L=1$ then we also obtain a degraded relative error of $7.8\cdot 10^{-4}$. The price of a poor choice of $L$ may be a significant increase in quadrature order so as to compensate for the loss of accuracy: If we choose an even smaller $L$ value such as $L=0.01$, $60$ function evaluations will give a relative error of  $4.9\cdot 10^{-3}$, and increasing the number of function evaluations to $90$ recovers an accuracy of $8.4 \cdot 10^{-4}$.
We do not discuss the interesting and important question of how best to choose $L$ further here,
since we have introduced this quadrature rule here only to illustrate our approach.

We now adapt the \textsf{B/CC}
quadrature rule in order to modify the objective function for \textsf{VF}.
In the $k$th step, the $r$th order rational approximant is defined as before:
$ H_r^{(k)}(s)=\frac{ \sum_{j=1}^r \frac{\phi_j^{(k)}}{s - \pol_j^{(k)}} }{1+ \sum_{j=1}^r \frac{\varphi_j^{(k)}}{s - \pol_j^{(k)}} }.$
The poles $\bflambda^{(k+1)}$ are determined from the $r$ roots of
$1+ \sum_{j=1}^r \frac{\varphi_j^{(k)}}{s - \pol_j^{(k)}}=0$.  Now, $\phi_j^{(k)}$ and  $\varphi_j^{(k)}$, will be determined from  the solution of the successive
weighted least squares problems
\begin{equation}
\|\Delta \left( \AC(\bflambda^{(k+1)}) x^{(k+1)}-h \right)\|_2\rightarrow\min,\;\; k = 0, 1, 2, \ldots,
\end{equation}
where
$x^{(k+1)} = \left(\begin{smallmatrix} \phi_1^{(k+1)} & \phi_2^{(k+1)} & \cdots & \phi_r^{(k+1)} & \varphi_1^{(k+1)} & \varphi_2^{(k+1)} & \cdots & \varphi_r^{(k+1)}\end{smallmatrix}\right)^T$,
\begin{multline}
\AC(\bflambda) = \left(\begin{smallmatrix}
\frac{1}{\xi_1-\pol_1} & \frac{1}{\xi_1-\pol_2} & \cdots & \frac{1}{\xi_1-\pol_r} & \frac{-H(\xi_1)}{\xi_1-\pol_1} &
\frac{-H(\xi_1)}{\xi_1-\pol_2} & \cdots & \frac{-H(\xi_1)}{\xi_1-\pol_r} \\[0.3em]
\frac{1}{\xi_2-\pol_1} & \frac{1}{\xi_2-\pol_2} & \cdots & \frac{1}{\xi_2-\pol_r} & \frac{-H(\xi_2)}{\xi_2-\pol_1} &
\frac{-H(\xi_2)}{\xi_2-\pol_2} & \cdots & \frac{-H(\xi_2)}{\xi_2-\pol_r} \\[0.3em]
\vdots & \vdots & \vdots & \vdots & \vdots & \vdots & \vdots & \vdots \\[0.3em]
\frac{1}{\xi_{\ell-1}-\pol_1} & \frac{1}{\xi_{\ell-1}-\pol_2} & \cdots & \frac{1}{\xi_{\ell-1}-\pol_r} & \frac{-H(\xi_{\ell-1})}{\xi_{\ell-1}-\pol_1} &
\frac{-H(\xi_{\ell-1})}{\xi_{\ell-1}-\pol_2} & \cdots & \frac{-H(\xi_{\ell-1})}{\xi_{\ell-1}-\pol_r}\\[0.3em]
\frac{1}{\xi_\ell-\pol_1} & \frac{1}{\xi_\ell-\pol_2} & \cdots & \frac{1}{\xi_\ell-\pol_r} & \frac{-H(\xi_\ell)}{\xi_\ell-\pol_1} &
\frac{-H(\xi_\ell)}{\xi_\ell-\pol_2} & \cdots & \frac{-H(\xi_\ell)}{\xi_\ell-\pol_r}\\[0.3em]
1 & 1 &  \ldots &  1 & 0 & 0 & \ldots & 0
\end{smallmatrix}\right), \quad
h = \left(\begin{smallmatrix} H(\xi_1) \\[0.4em] H(\xi_2) \\[0.4em] \vdots \\[0.4em] H(\xi_{\ell-1}) \\[0.4em]  H(\xi_\ell) \\[0.4em] M_{+}[H] \end{smallmatrix}\right), \\[1em]
 \mbox{and } \Delta = \mathrm{diag}\left( {\rho_1},\, {\rho_2},\, \ldots,\,{\rho_{\ell}},\,{\rho_{+}}\right)  \mbox{ with nodes }\xi_j=\imunit L \cot\left(\frac{j\pi}{\ell+1}\right) \\
\mbox{ and weights }\rho_j=\csc\left(\frac{j\pi}{\ell+1}\right)\,\sqrt{\frac{L\pi}{(\ell+1)} }\;\mbox{ for }j=1,\ldots, \ell
\mbox{ and }\rho_{+}=\sqrt{\frac{\pi}{L(\ell+1)}}
\end{multline}
determined by the quadrature rule (\ref{eq:Boyd:formula}).  This describes the main iteration of our quadrature-based
variant of \textsf{VF}. We will refer to this variant as \textsf{QuadVF}.   The term $M_{+}$ from (\ref{first_last_terms}) is retained and given double weight, since $M_{+}[H]=M_{-}[H]$ for real systems.
Notice that the weighting matrix $\Delta$ is fixed with respect to $k$
and that the quadrature nodes are closed under conjugation: $\xi_j=\overline{\xi_{\ell+1-j}}$, halving the number of function evaluations needed to implement the formula.  This symmetry is also reflected in the weights:
$\rho_j=\rho_{\ell+1-j}$.

\subsection{Numerical Comparisons}\label{SS=NumericalComparison}
\subsubsection{Heat Model: \textsf{VF} vs. \textsf{QuadVF}}
We use the aforementioned  Heat Model for this example.
We  take $\ell=20$ samples  (requiring only $10$ function evaluations due to the complex conjugate sampling points) and apply both \textsf{VF} and \textsf{QuadVF}  to construct order  $r=4$ rational approximants.  In this case, the sampling nodes for both \textsf{VF} and \textsf{QuadVF} nodes are contained in $\imunit [2.7705 \times 10^{-2},2.4527]$;
only the distribution of the nodes is different.  The resulting relative $\Hardy_2$ error norms are $8.4776 \times 10^{-1}$ for \textsf{VF} and $6.9326 \times 10^{-3}$ for \textsf{QuadVF}.
The numbers for the relative $\Hardy_\infty$ error norms were even more revealing: $1.6392$ for \textsf{VF} and
  $6.7765 \times 10^{-4}$ for the \textsf{QuadVF}.

  Note that the poor approximation resulting from \textsf{VF} is not due to a
  large residual for the underlying LS problem. On the contrary,  \textsf{VF} leads to a relative LS residual norm of  $1.8943 \times 10^{-3}$, representing a very accurate solution to the discrete LS problem; for \textsf{QuadVF},  the relative residual norm is   $3.5430 \times 10^{-4}$,  yielding in this case not only an accurate solution to the discrete LS problem but also a comparable level of accuracy as an ideal $\Hardy_2$-optimal reduced model of the same order.   \textsf{VF} does a great job in minimizing the least-squares error over the given samples; however the samples are local in nature and do not reflect the global $\Hardy_2$ and/or $\Hardy_\infty$  behavior. By choosing the sampling nodes from an appropriate quadrature rule, the discrete error that is minimized becomes a much better approximation to the true  $\Hardy_2$  error,  leading ultimately to a better rational approximation.

\subsubsection{FOM Model: \textsf{VF} vs \textsf{QuadVF}}
We repeat the same numerical experiments for the FOM Model
by taking $\ell=50$ samples (requiring only $25$ function evaluations) and applying \textsf{VF} and \textsf{QuadVF} as before. For this model,
we construct an order  $r=12$ rational approximant. The sampling interval for  \textsf{VF} and \textsf{QuadVF} is the same: $\xi \in \imunit [3.0810,\,1.6213\times 10^3]$, again differing only by their distribution in the interval.
The resulting relative $\Hardy_2$ error norms are:  $3.0903 \times 10^{-2}$ for \textsf{VF}, and $1.8561 \times 10^{-3}$ for \textsf{QuadVF};  \textsf{QuadVF} outperforms \textsf{VF}  by more than an order of magnitude in terms of accuracy.  Similar results are found for $\Hardy_\infty$ performance as well with \textsf{VF}
and  \textsf{QuadVF} leading to relative $\Hardy_\infty$ error norms of, respectively,
  $7.6430 \times 10^{-2}$  and
  $3.2204 \times 10^{-3}$.
  As in the previous example, the difference in the approximation quality is not due to the underlying discrete LS residuals.
  Both  \textsf{VF} and \textsf{QuadVF} produced very accurate LS solutions with relative residual norms of  $8.1809 \times 10^{-5}$ and $4.6997 \times 10^{-5}$, respectively.  The improved  node and weight  selection of \textsf{QuadVF} appears to be the determining factor for the improved  quality of the rational approximation.
However, even \textsf{QuadVF} does not match the high-fidelity optimal rational approximations. For this example, \textsf{IRKA} produces final reduced models with relative $\Hardy_2$ and $\Hardy_\infty$  errors of
$1.9200 \times 10^{-4}$ and  $2.1157\times 10^{-4}$, respectively; an order of magnitude better in both cases.

\subsubsection{Heat Model: \textsf{QuadVF} vs \textsf{IRKA}} \textsf{QuadVF} is based on the discretization of  the true $\Hardy_2$ norm. Therefore in this example, we investigate numerically how the solution of the quadrature-based discrete $\Hardy_2$ minimization problem compares to the the solution of the continuous  $\Hardy_2$ problem by \textsf{IRKA} as the number of sampling points $\ell$ increases. We use the Heat Model and construct order $r=2$ rational approximants using   \textsf{QuadVF}  and \textsf{IRKA}. Let $H$, $H_1$,$H_2$ denote, respectively, the full-order model, the reduced model by \textsf{IRKA} and the reduced model by \textsf{QuadVF} .
In Table \ref{table:quadvf_vs_irka} below, we list the relative $\Hardy_2$ distances between $H_1$ and $H_2$ as
$\ell$ increases in addition to
the relative  $\Hardy_2$ distances between the full and two reduced models:

\begin{table}[hhh]
\centering
\begin{tabular}{|c||c|c|c|}
\hline
 $\ell$& $\displaystyle{\frac{\| H_1 - H_2\|_{\Hardy_2}}{\| H_1 \|_{\Hardy_2}}}$ & $\displaystyle{\frac{\| H - H_1\|_{\Hardy_2}}{\| H \|_{\Hardy_2}}}$ & $\displaystyle{\frac{\| H - H_2\|_{\Hardy_2}}{\| H \|_{\Hardy_2}}}$   \\
 \hline\hline
 $10$ &  $1.1919\times 10^{-2}$&   $3.9483\times 10^{-2}$& $4.1348\times 10^{-2}$\\ \hline
  $100$ &  $3.8795\times 10^{-3}$& $3.9483\times 10^{-2}$ &$ 3.9681\times 10^{-2}$\\ \hline
   $1000$ & $1.0239\times 10^{-3}$ &$3.9483\times 10^{-2}$ &$3.9497 \times 10^{-2}$\\ \hline
    $5000$ &  $5.2313\times 10^{-4}$& $3.9483\times 10^{-2}$& $3.9487\times 10^{-2}$\\ \hline
     $15000$ & $4.4926\times 10^{-4}$ &$3.9483\times 10^{-2}$&$ 3.9486\times 10^{-2}$\\ \hline
\end{tabular}
\label{table:quadvf_vs_irka}
\caption{Relative $\Hardy_2$ distances vs $\ell$}
\end{table}
Table  \ref{table:quadvf_vs_irka}  illustrates that for this numerical example, as $\ell$ increases, the solution of the discrete $\Hardy_2$ problem via \textsf{QuadVF} is converging to the true $\Hardy_2$ solution. This is an encouraging result confirming that an effective quadrature-based selection for the discretized $\Hardy_2$ problem  might yield rational approximants close to those of the true, continuous problem. These issues will be
further studied and presented in \cite{Beattie-Drmac-Gugercin:2013-QIRKA}. For comparison, we increased the sampling size for the \textsf{VF}  as well. However, even with $\ell = 15000$, \textsf{VF}  produced a rational approximant, $H_3(s)$, with relative $\Hardy_2$ distances
$$
\displaystyle{\frac{\| H_1 - H_3\|_{\Hardy_2}}{\| H_1 \|_{\Hardy_2}}} = 9.8470 \times 10^{-1}\quad\mbox{and}\quad
\displaystyle{\frac{\| H - H_3\|_{\Hardy_2}}{\| H_3 \|_{\Hardy_2}}} = 9.8503 \times 10^{-1},
$$
The contrast with \textsf{QuadVF} underscores the value of sampling guided by an effective quadrature rule.

\subsubsection{ISS1R Module: \textsf{QuadVF} vs \textsf{VF}}
We use the ISS 1R module \cite{gugercin2001iss} with $n=270$ and approximate it with a model of order $r=16$. We first use \textsf{QuadVF} and  $25$ function evaluations ($\ell=50$ nodes in $25$ complex conjugate pairs).
The relative $\Hardy_2$  and $\Hardy_\infty$  errors of \textsf{QuadVF} were $7.2156\times 10^{-2}$ and $2.4448\times 10^{-2}$. The relative $\Hardy_2$  and $\Hardy_\infty$  errors of \textsf{IRKA} (using the same initial
poles as \textsf{QuadVF}) were, respectively, $1.4474\times 10^{-2}$ and
$5.5595 \times 10^{-3}$ -- lower, as expected.
Next, for comparisons, we use the same interval
$\imunit[1.2324\times 10^{-1}, 6.4853\times 10^{1}]$
containing the quadrature nodes, and replace the nodes by
the same number of \emph{(i)} linearly spaced points, and \emph{(ii)} logarithmically spaced points.
Then \textsf{VF}
is run with those points.  For the case of linearly spaced points, \textsf{VF} produced relative
$\Hardy_2$ and $\Hardy_\infty$  errors of $104.24\%$ and $99.79\%$, respectively, almost two orders of magnitude higher errors than \textsf{QuadVF}. For logarithmically spaced points, \textsf{VF} performed better and produced relative $\Hardy_2$ and $\Hardy_\infty$  errors $3.2872 \times 10^{-1}$ and $1.2257 \times 10^{-1}$; still much less accurate than \textsf{QuadVF}.  The Bode plots of the full-model and all four rational approximants are shown in Figure \ref{FIG:SS-VFQ}.

 \begin{figure}
\begin{center}
 \includegraphics[width=2.95in,height=2.1in]{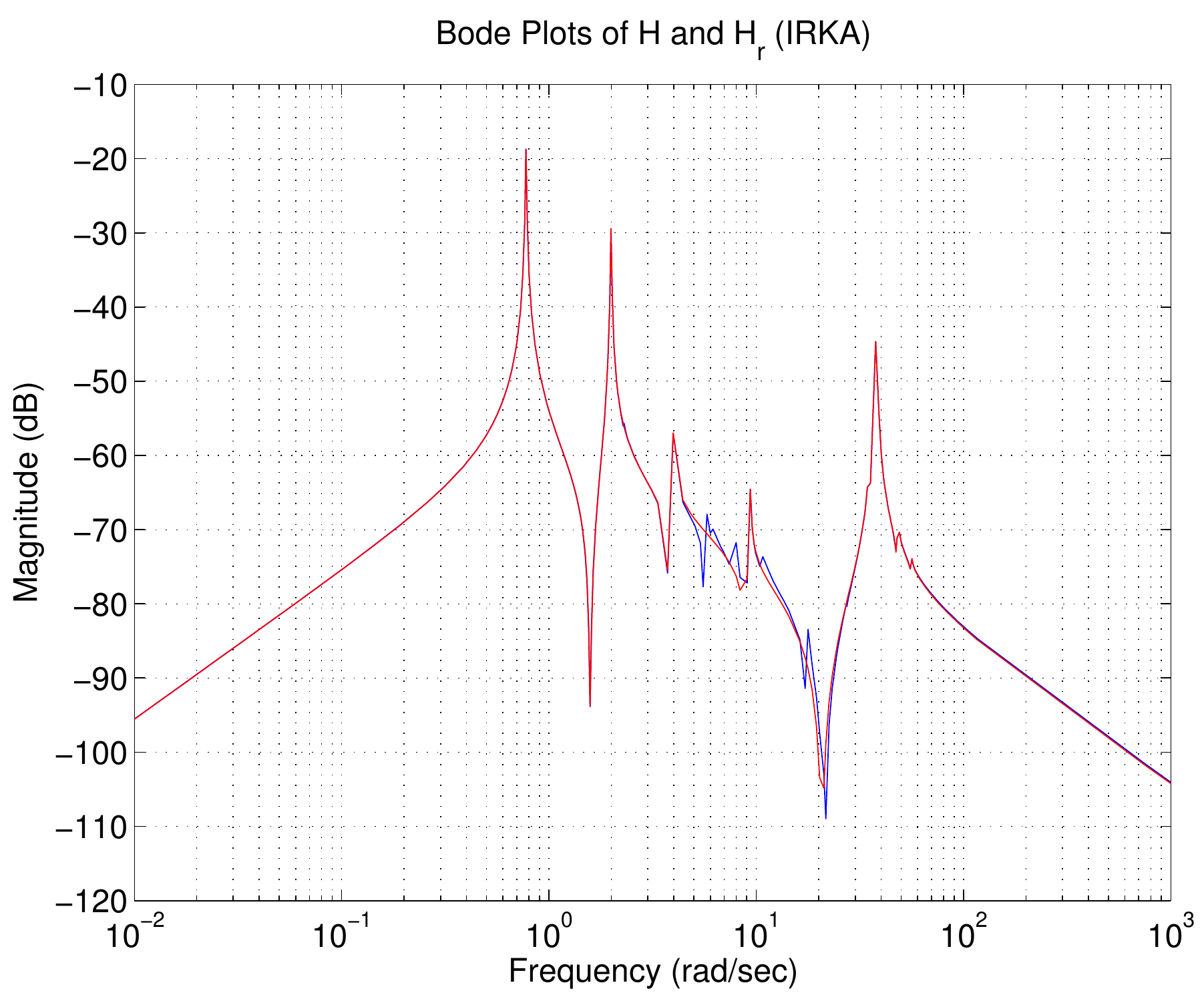}
 \includegraphics[width=2.95in,height=2.1in]{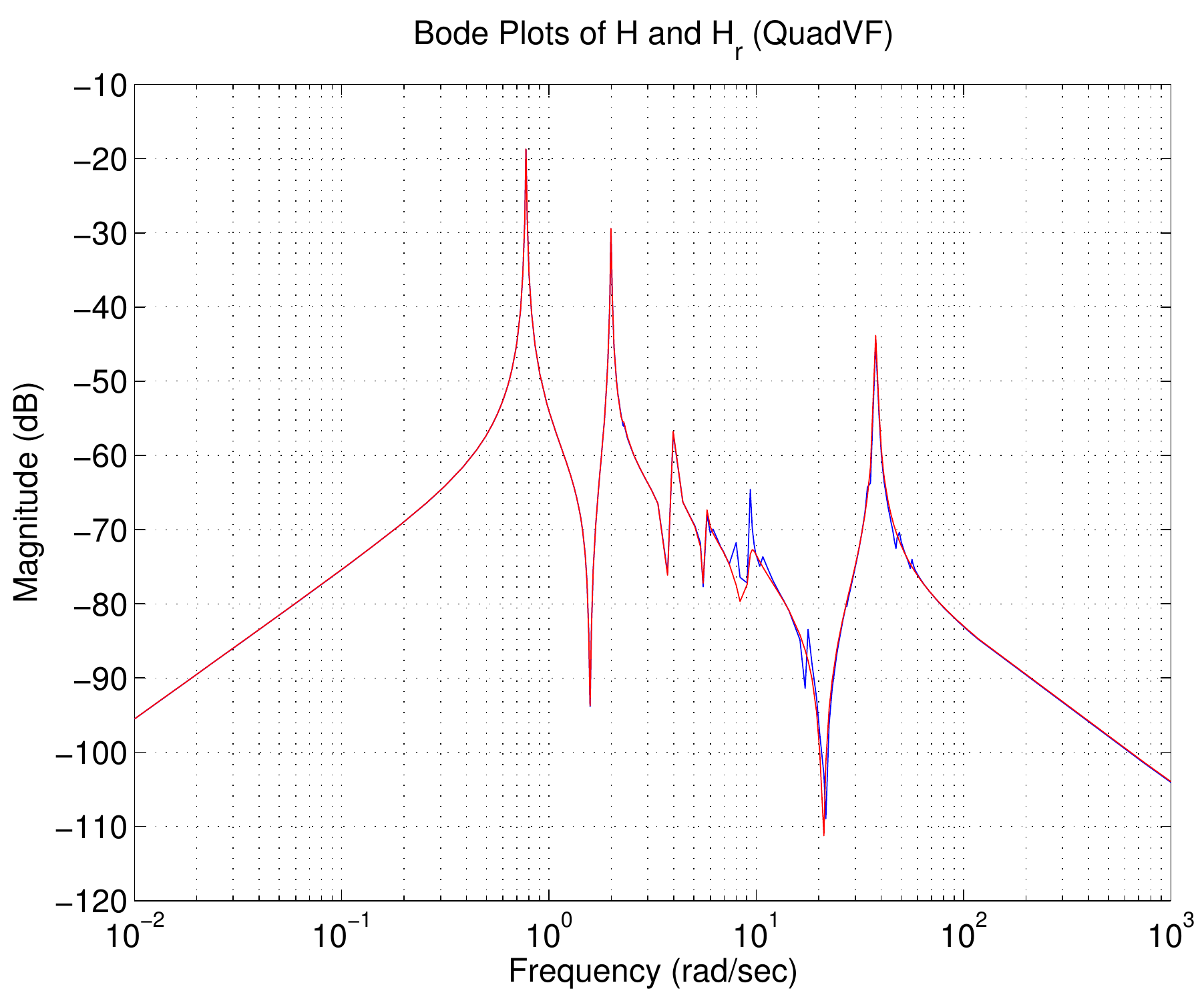}
 
 \includegraphics[width=2.95in,height=2.1in]{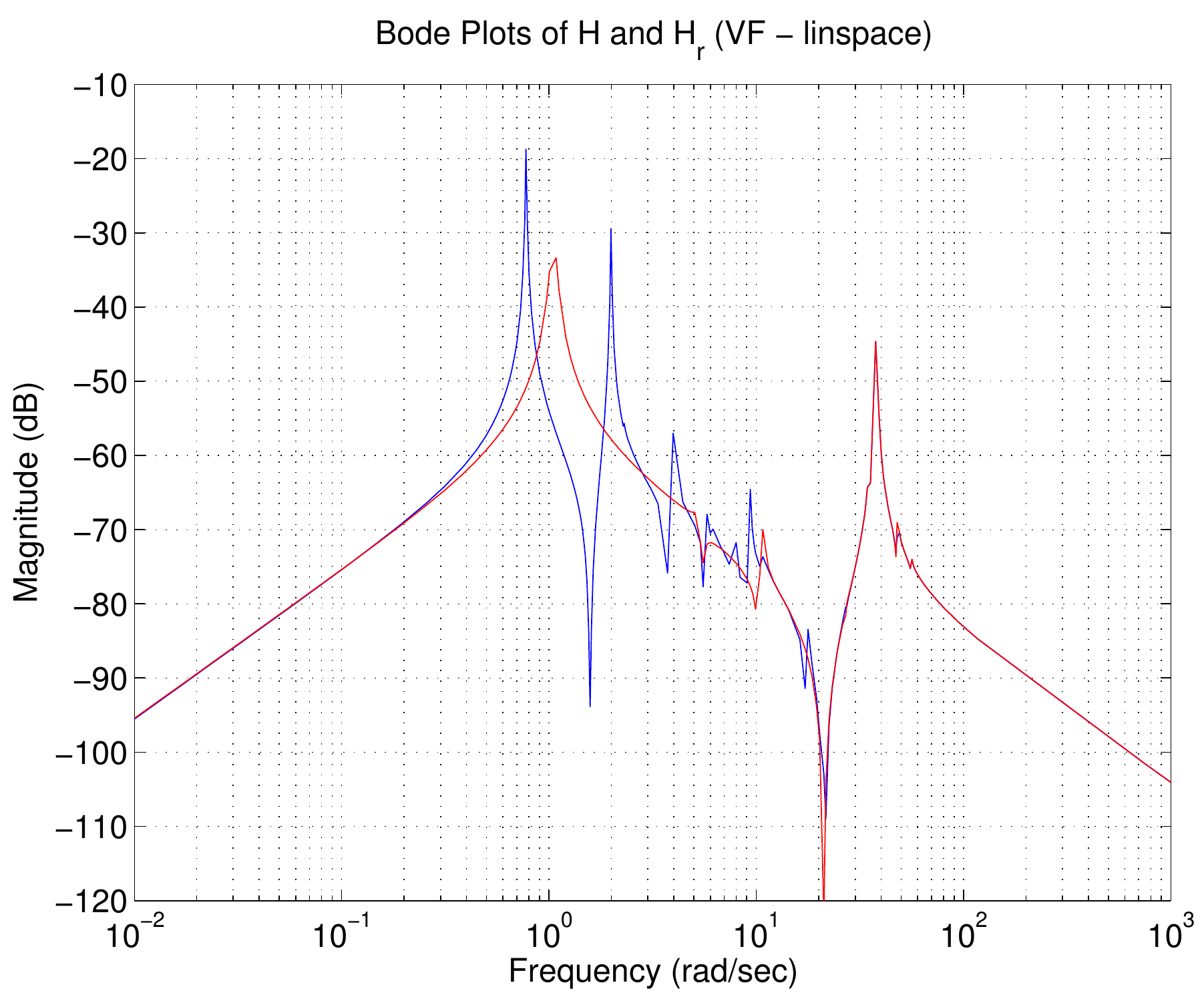}
 \includegraphics[width=2.95in,height=2.1in]{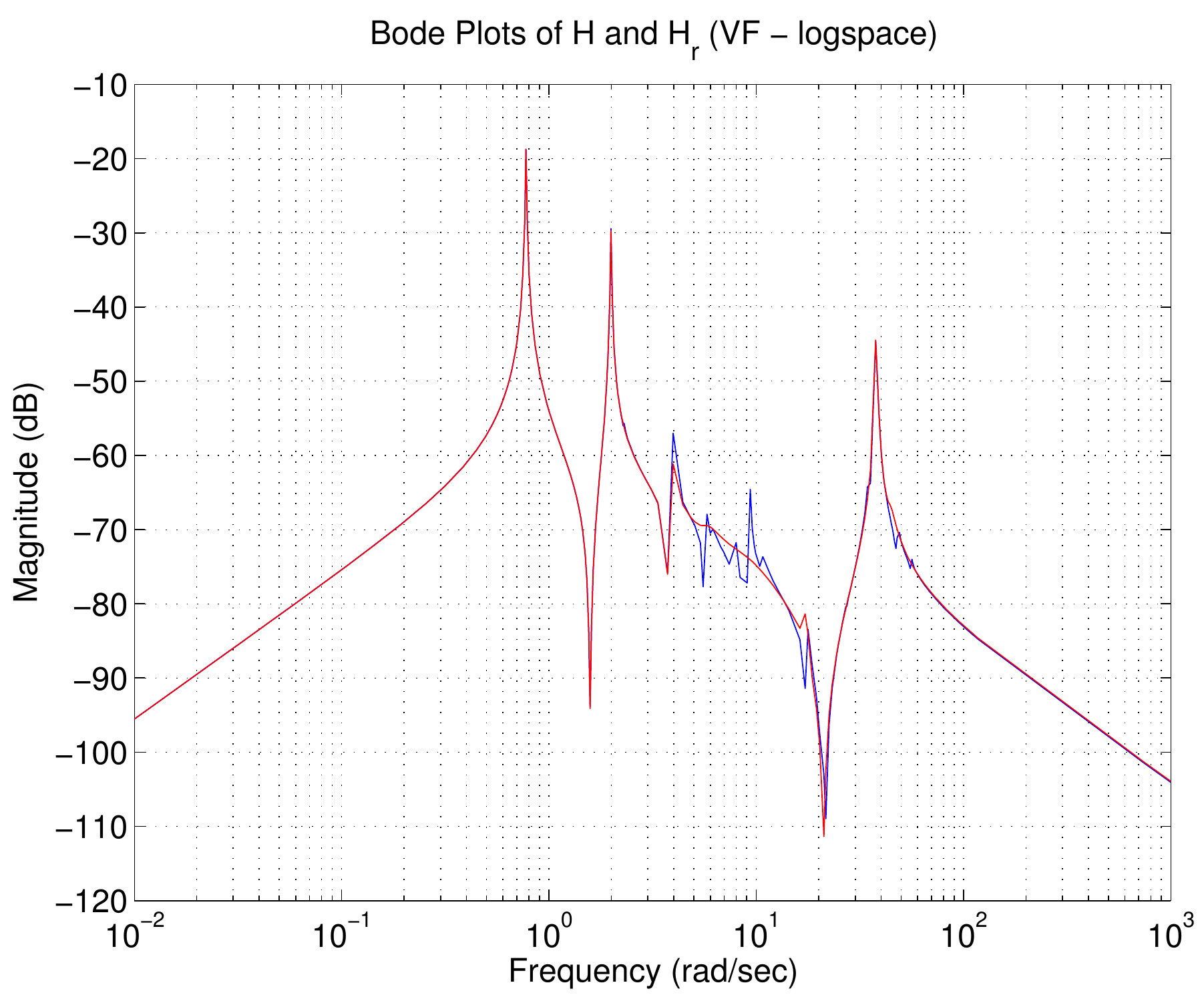}
 \end{center}

  \caption{\label{FIG:SS-VFQ} Amplitude Bode plots of the original system (blue line in every plot) and four rational approximations (red line) (Top-left: \textsf{QuadVF}, Top-right: \textsf{IRKA}, Bottom-left: \textsf{VF} with linearly spaced points, Bottom-right: \textsf{VF} with logarithmically spaced points.)}
\end{figure}

\begin{remark}
{\em
Recently, Hochman, Leviatan and White \cite{Hochman-Leviatan-White-2012} also formulated rational least squares approximation
using the information from the quadrature nodes. There, the problem is to find real valued potential $U$ that satisfies Laplace equation
in a simply connected domain $\Omega\subset \mathbb{R}^2$ and the Dirichlet boundary condition $U_{|_\Gamma}=f$ on the boundary curve $\Gamma$ of $\Omega$.
The idea is to approximate  $U$ with  the truncated real part $\hat{U}$ of a weighted sum $W$ of complex dipole potentials, and to enforce the boundary condition
on $\Gamma$ by minimizing $\|\hat{U} - f\|_\Gamma$, where $\|\cdot\|_\Gamma$ is induced by the inner product
$(u,v)_\Gamma = \int_0^1 u(z(s))v^*(z(s))\lambda(s) ds$ along $\Gamma$.
(Here $z(s)$ is a parametrization of $\Gamma$ and $\lambda(s)$ is a positive weight function.) Discretizing the norm introduces
the quadrature nodes.
}
\end{remark}

\subsection{Vector fitting in a discrete Sobolev norm}\label{SS=Sobolev-VF}
Incorporating derivative information into function approximation strategies (e.g., by penalizing roughness of the error function, or forcing Hermite interpolation at selected points) often can produce significantly higher fidelity approximations at only marginally increased cost.  Many interpolatory model reduction methods, including \textsf{IRKA}, construct rational approximants, $H_r(s)$, that match the value of $H(s)$ together with some of its derivatives at selected interpolation points, a type of generalized Hermite interpolation.  Since derivatives in the frequency domain are associated with moments in the time domain, the expression ``moment matching methods",
as exemplified e.g., by the ``Pad\'{e} via Lanczos" (PVL) method  \cite{Feldmann:2006:ELC:2298628.2302601}, refers also to a similar generalized Hermite interpolation strategy.

Chen, Zheng, and Fang \cite{Chen-Zheng-Fang-2003}
included derivatives in their modification of \textsf{VF},
leading to what they termed  ``Moment Matching Vector Fitting",  a
multipoint moment matching scheme with the approximating rational function given in barycentric form.
Derivative conditions that are compatible with the \textsf{VF} framework can be obtained by differentiating the expression, $H(s) d(s) = n(s)$. For example, to match the first derivative,
 one uses the condition $d'(s)H(s) + d(s)H'(s)-n'(s)=0$, which is a linear expression
 in the coefficients of $n(s)$ and $d(s)$.
Based on this expression and similar ones for higher derivatives,  Chen, \textit{et al.} in \cite{Chen-Zheng-Fang-2003} derived a system of equations that incorporate derivative conditions. The assumed barycentric form of the approximant then produces a coefficient matrix with a Cauchy-like structure similar to what is obtained for \textsf{VF}.

In this section, we develop a somewhat different approach toward incorporating derivative information into  \textsf{VF}.  Analogous to our approach for \textsf{QuadVF}, we begin with an approximation problem formulated with respect to an appropriate continuous norm and then discretize, making use of effective quadrature points and weights.
Derivative conditions arise differently than in \cite{Chen-Zheng-Fang-2003}, leading to a significant difference in the diagonal scaling.

Given $H(s)$ and sampling nodes, $\xi_i$, we seek a rational function, $H_r(s)$, that will yield good approximations not only to  $H(\xi_i)$ but also to $H'(\xi_i)$, in the least-squares sense. Restated formally, the problem is to find an
$r$th order stable rational approximant:
\begin{equation}\label{eq:VF-Hermite}
\begin{array}{c}
H_r(s) = \frac{n(s)}{d(s)} \equiv \frac{{\sum_{j=1}^r \frac{\phi_j}{s - \pol_j}}}
{{ \sum_{j=1}^r \frac{\varphi_j}{s - \pol_j}+1}},\\[.2in]
\mbox{such that }\;\; \sum_{i=1}^\ell (\rho_{i0}^2
| H_r(\xi_i) - H(\xi_i)|^2 +\rho_{i1}^2
| H'_r(\xi_i) - H'(\xi_i)|^2) \longrightarrow \min.
\end{array}
\end{equation}
There is a significant difference in our problem formulation (\ref{eq:VF-Hermite}) and that of \cite{Chen-Zheng-Fang-2003}.
We view the minimization problem considered in (\ref{eq:VF-Hermite}), as the discretization of a minimization problem formulated now with respect to a continuous Sobolev-type $\Hardy_2$ norm,
$$
\big\|H - H_r\big\|^2_{\star} = \big\| H - H_r\big\|^2_{\Hardy_2} +  \big\| H' - H'_r\big\|^2_{\Hardy_2},
$$
and apply an appropriate quadrature rule (see e.g. \cite{Kim:2002:QuadratureFirstDerivatives}) to determine nodes $\xi_i$ and weights $\rho_{i0}$, $\rho_{i1}$ in (\ref{eq:VF-Hermite}).  This has the effect of penalizing roughness of the error function, $H - H_r$, and will yield a different rational approximant to $H(s)$.  For an overview of derivative-weighted least squares approximation, we refer to \cite[\S 3.2.3]{Gautschi-2004}.

To arrive at a \textsf{VF} iteration for (\ref{eq:VF-Hermite}), first approximate the derivative error
\begin{align}
H'(s) - H_r'(s) =& H'(s)+ \frac{n(s)d'(s)-n'(s)d(s)}{d^2(s)} = \frac{d(s)H'(s)+d'(s)H_r(s)-n'(s)}{d(s)} \nonumber \\
& \approx \frac{d(s)H'(s)+d'(s)H(s)-n'(s)}{d(s)}.  \label{derivApprox}
\end{align}

Then approximating the $\Hardy_2$ norms with quadrature rules and incorporating the rescaling characteristic of the \textsf{SK} iteration produces a weighted LS problem that appears as
\begin{align*} \label{discSobErr}
\big\|H - H_r\big\|^2_{\star} \approx & \sum_{i=1}^\ell   \frac{\rho_{i0}^2}{|d^{(k)}(\xi_i)|^2} \left| \sum_{j=1}^r \frac{\phi_j^{(k+1)}}{\xi_i-\pol_j^{(k)}}  - \sum_{j=1}^r \frac{H(\xi_i)}{\xi_i-\pol_j^{(k)}} \varphi_j^{(k+1)} -H(\xi_i) \right|^2  \\
+\sum_{i=1}^\ell &  \frac{\rho_{i1}^2}{|d^{(k)}(\xi_i)|^2} \left|
\sum_{j=1}^r \frac{-\phi_j^{(k+1)}}{\left(\xi_i-\pol_j^{(k)}\right)^2}  + \sum_{j=1}^r \left( \frac{H(\xi_i)}{\left(\xi_i-\pol_j^{(k)}\right)^2}
- \frac{H'(\xi_i)}{\xi_i-\pol_j^{(k)}}\right) \varphi_j^{(k+1)} -H'(\xi_i)\right|^2 .
\end{align*}

The structure of the LS matrix (cf.(\ref{eq:LS})-(\ref{eq:A=PC})) becomes more complicated:
Set $D_\xi' = \mathrm{diag}(h')$, $h'=(H'(\xi_i))_{i=1}^\ell$, $W_j=\mathrm{diag}(\rho_{ij})_{i=1}^\ell$, ($j=0,1$),
$\Delta^{(k)}=\mathrm{diag}(1/|d^{(k)}(\xi_i)|)_{i=1}^\ell$, and $\mathcal{C}^{(k)}_{ij}=1/(\xi_i-\lambda_j^{(k)})$.
 The new LS
problem reads
\begin{equation}\label{eq:Sobolev-ls-scaled}
\left\| \begin{pmatrix} W_0 \Delta^{(k)} & 0 \cr 0 & {W}_1 \Delta^{(k)}\end{pmatrix}
\left\{ \begin{pmatrix} \mathcal{C}^{(k)} & -D_\xi \mathcal{C}^{(k)} \cr - (\mathcal{C}^{(k)}\circ\mathcal{C}^{(k)}) &
 D_\xi (\mathcal{C}^{(k)}\circ\mathcal{C}^{(k)}) - D_\xi'\mathcal{C}^{(k)}\end{pmatrix}
 \begin{pmatrix} \phi_{1:r}^{(k+1)} \cr \varphi_{1:r}^{(k+1)}\end{pmatrix} - \begin{pmatrix} h \cr h' \end{pmatrix} \right\}\right\|_2 \rightarrow \min
\end{equation}
where ``$\circ$" denotes the
Hadamard matrix product.  

The final expression of (\ref{derivApprox}) is approximate because a correction term, 
$\frac{d'(s)}{d(s)}(H(s)-H_r(s))$, has been dropped.  
This additional term may be retained and incorporated into the final LS problem  (\ref{eq:Sobolev-ls-scaled}), although the additional complexity might not be justified. 
For example, one may approximate the correction term evaluated at $s=\xi_i$ as 
$$
\frac{d'(\xi_i)}{d(\xi_i)}(H(\xi_i)-H_r(\xi_i))\approx \frac{d^{(k+1)}(\xi_i)}{d^{(k)}(\xi_i)}\left(H(\xi_i)- \frac{n^{(k)}(\xi_i)}{d^{(k)}(\xi_i)}\right).
$$
This yields a more complicated, though similarly structured LS coefficient matrix.
We believe that this is not necessary in practice since the effect of penalizing derivative error appears to be achieved quite effectively with the simpler expression.  
 Note that the first part of the Sobolev error expression, $\|H - H_r\|^2_{\star}$, penalizes the magnitude of $H(\xi_i)-H_r(\xi_i)$ suggesting that the correction term that has been omitted will become small in any case. In addition, as the iteration progresses, the residues of $d(s)$ are expected to converge to  $0$, so that $d(s)\rightarrow 1$ and 
 $d'(s)\rightarrow 0$ almost everywhere, further diminishing the term that has been omitted. 

Adopting the pole relocation and rescaling strategies characteristic of \textsf{VF}, we find
\begin{proposition}\label{Prop:Sobolev-VF}
By a change of barycentric representation, the LS problem (\ref{eq:Sobolev-ls-scaled}) can be replaced by
\begin{equation}
\left\| \begin{pmatrix} W_0  & 0 \cr 0 & {W}_1 \end{pmatrix}\!\!
\left\{ \begin{pmatrix} \mathcal{C}^{(k+1)} & -D_\xi \mathcal{C}^{(k+1)} \cr - (\mathcal{C}^{(k+1)}\circ\mathcal{C}^{(k+1)}) &
 D_\xi (\mathcal{C}^{(k+1)}\circ\mathcal{C}^{(k+1)}) - D_\xi'\mathcal{C}^{(k+1)}\end{pmatrix}
 \begin{pmatrix} \tilde{\phi}_{1:r}^{(k+1)} \cr \tilde{\varphi}_{1:r}^{(k+1)}\end{pmatrix} - \begin{pmatrix} h \cr h' \end{pmatrix} \right\}\right\|_2 \rightarrow \min,
\end{equation}
where $\mathcal{C}^{(k+1)}_{ij}=1/(\xi_i-\lambda_j^{(k+1)})$, and $(\lambda_j^{(k+1)})_{j=1}^\ell$ are
the zeros of $d^{(k)}(s)$.
\end{proposition}
\begin{proof}
Consider all iterations done up through step $k+1$ to have been done with fixed poles, namely $\lambda_j^{(k+1)}$
for $j=1,\ldots,\ell$.
If we want the next iterate to be represented in the barycentric form with the nodes $\lambda_j^{(k+1)}$, then, to be consistent
with the definition of the iterations (\ref{eq:relaxedNLS}), the scaling factors $1/|{d}^{(k)}(\xi_i)|$ must be computed
using the barycentric form of $H_r^{(k)}=n^{(k)}/d^{(k)}$ based on the nodes $\lambda_j^{(k+1)}$.
Now, if we represent ${n}^{(k)}/{d}^{(k)}$,
with ${d}^{(k)}$ as in (\ref{PoleZero_d_k}), with the nodes $\lambda_j^{(k+1)}$, then we obtain
$
\frac{{n}^{(k)}(s)}{{d}^{(k)}(s)} = \frac{\sum_{j=1}^r \frac{\widehat{\phi}_j^{(k+1)}}{s - \pol_j^{(k+1)}}}{1} .
$
Hence, in this representation the scaling factors are $1$.
\end{proof}

The Sobolev norm-based \textsf{VF} iteration described in Proposition \ref{Prop:Sobolev-VF} will be called \textsf{SobVF} and will be run typically until the nodes $\lambda_j^{(k)}$ converge (numerically) at some index $k_*$.
To compute our final rational  approximant, we take the converged  $\lambda_j^{(k_*)}$'s as the poles and solve
LS problem
\begin{equation}\label{eq:Sobolev-VF-residues}
\left\| \begin{pmatrix} W_0  & 0 \cr 0 & {W}_1 \end{pmatrix}\!\!
\left\{ \begin{pmatrix} \mathcal{C}^{(k_*)}  \cr - (\mathcal{C}^{(k_*)}\circ\mathcal{C}^{(k_*)}) \end{pmatrix}
{\phi}_{1:r}  - \begin{pmatrix} h \cr h' \end{pmatrix} \right\}\right\|_2 \rightarrow \min
\end{equation}
the compute the final residues $\phi_j$.
\begin{remark}
{\em
Even though obtaining the derivative information may not be always feasible (e.g., in the data driven setting), in many cases $H'(s)$ can be computed without much additional cost. For example, if a state space representation $H(s)=\bfC(s\mathbf{I}-\mathbf{F})^{-1}\bfB$
is available, then computing $H'(s) = - \bfC(s\mathbf{I}-\mathbf{F})^{-2}\bfB$ is not expensive if the function evaluation is performed using, for example, sparse direct
solvers or a Hessenberg decomposition-based method for dense computations \cite{Beattie_Drmac_Gugercin:2011_FreqResp}. The evaluation of $H(s)$ already requires
the computation of a decomposition of $(s\mathbf{I}-\mathbf{F})$ at the node $s=\xi_i$. Since evaluating $H'(s)$ at the node
$s=\xi_i$ requires solving a linear system with the same coefficient matrix, the triangular factors can be reused, and $H(s)$ and $H'(s)$ at the node $\xi_i$
are obtained with only small additional cost.
}
\end{remark}

\subsubsection{Numerical Examples for \textsf{SobVF}}
We illustrate the effectiveness of \textsf{SobVF} using two models from the
\textsf{NICONET}  Benchmark Collection, comparing results with \textsf{VF}.
Since \textsf{SobVF} uses both $H(s)$ and $H'(s)$ at the sampling nodes, we use
twice the number of nodes in \textsf{VF} in order to present a fair comparison for \textsf{VF}; that is, if we use $\ell$ nodes in (\ref{eq:VF-Hermite}), we will employ $2\ell$ in \textsf{VF}.
For brevity, instead of adapting and giving details of a Hermite quadrature rule, we simply use the weights and the nodes of the Clenshaw-Curtis formula from \S \ref{SS=SamplingviaQuadrature} in both examples.

\begin{example}\label{Example:Sobolev-VF}
{\em
The first example is the Building Model from the \textsf{NICONET} benchmark collection with order $n=48$. We have chosen this model since it is very hard to approximate and a high-fidelity approximation is achieved only for large $r$ values \cite{ASG01}. For example, to reach a relative $\Hardy_2$ error norm of $10^{-4}$, even the optimal rational approximation method \textsf{IRKA}
requires $r=40$ and then yields a relative $\Hardy_2$ of
$1.18 \times 10^{-4}$.  We pick $r=40$ and obtain the nodes and weights using \S \ref{SS=SamplingviaQuadrature}.
The range of nodes for \textsf{VF} and \textsf{SobVF} is the same; only the distribution is different.  For $\ell=25$, \textsf{VF} using $2\ell = 50$ logarithmically spaced nodes yields a relative $\Hardy_2$ error norm of $1.564$ -- quite a poor approximation. On the other hand,  using  \textsf{SobVF} as in (\ref{eq:VF-Hermite})
with $\ell = 25$ nodes yields a rational approximant with a relative $\Hardy_2$ error of $6.56\cdot 10^{-3}$.  This constitutes a three order-of-magnitude improvement over what \textsf{VF} provides without greater computational cost; recall  \textsf{VF} used twice the number of nodes as  \textsf{SobVF}.
}
\end{example}
\begin{example}  \label{Example:Sobolev-VF-Beam}
{\em
We consider the Beam Model for the \textsf{NICONET} benchmark collection.
This model has order  $n=348$. Using $\ell=25$ as in the previous example for \textsf{SobVF} approximation
and $2\ell = 50$ nodes
for \textsf{VF} approximation, we
obtain relative $\Hardy_2$ errors of $1.29$ for \textsf{VF} and $0.16$ for \textsf{SobVF}.
To obtain better approximants, we double the number of nodes to $\ell = 50$, leading to a
relative $\Hardy_2$ error norm of  $4.84\cdot 10^{-2}$ for \textsf{VF} and
and $2.85\cdot 10^{-4}$ for \textsf{SobVF}.  We observe that
for $r=40$,  the optimal  approximation method  \textsf{IRKA} yield a relative error of $2.09 \cdot 10^{-4}$. 
So, using $\ell=50$ nodes, \textsf{SobVF} very nearly achieves the accuracy captures the accuracy of 
a locally optimal approximant. To investigate how the approximants change, we increase the order to $r=70$. 
Curiously, this caused a \emph{higher} relative error of $1.84\cdot 10^{-1}$ for \textsf{VF}. 
This is mainly due to the numerical ill-conditioning of the underlying LS problem induced by increasing $r$. 
These issues are explained in more detail in \S \ref{SS=CaveatConditioning}.
On the other hand, increasing $r$ to $70$ had no apparent adverse effect on the \textsf{SobVF};
 the relative error decreased to $4.17\cdot 10^{-6}$. For comparison, note that
for $r=70$,  the relative $\Hardy_2$ error produced by \textsf{IRKA} is $5.10\cdot 10^{-7}$. Although
\textsf{IRKA} is still better (as expected),  the \textsf{SobVF} 
approximation is achieving close to the same accuracy.
}
\end{example}

In both of the experiments described above, the \textsf{SobVF} approximation was substantially more accurate than
 a \textsf{QuadVF} approximation produced with the same set of nodes and weights. 
 As previously stated, this will not even be the best performance that can be expected from
\textsf{SobVF}. The full-potential of (\ref{eq:VF-Hermite})
will be realized once we adopt an appropriate quadrature rule,
much as we did in \S \ref{SS=SamplingviaQuadrature} to produce \textsf{QuadVF} .
We defer these considerations to a later time.

\section{Practical Issues} \label{sec:prac_issues}

We focus on the convergence behavior and some practical issues impacting the numerical
implementation of both \textsf{VF} and \textsf{QuadVF}.

\subsection{Unstable nodes mirroring and scaling}\label{SSS::Mirroring-and-scaling}
One of the advantages of the pole relocation step in  \textsf{VF} is that the emergence of unstable poles can be resolved
and the iterates can be steered to a stable approximant. This is achieved by reflecting those unstable nodes  (poles) that are in $\Cplx_+$  with respect to
the imaginary axis and placing them in  $\Cplx_-$. The same procedure is also employed in \textsf{IRKA}.
 Let  $\widetilde{n}^{(k)}(s)/\widetilde{d}^{(k)}(s)$ be the current approximation,  $\lambda_j^{(k+1)}$ denote the the originally computed set of zeros of $\widetilde{d}^{(k)}$
and  $\lambda_{j_t}^{(k+1)}$, $t=1,\ldots,p,$ be the $p<r$ of these poles that are in $\Cplx_+$. Then, \textsf{VF} replaces $\lambda_{j_t}^{(k+1)}$
with $-\lambda_{j_t}^{(k+1)}$  while keeping the remaining stable ones as is to obtain the new set of poles,
to be denoted by $\widehat{\lambda}_j^{(k+1)}$ {with $\widehat{\lambda}_t^{(k+1)}=-\lambda_{j_t}^{(k+1)}$, $t=1,\ldots,p$}.
From a systems theoretic perspective, the mirroring of an unstable pole $\lambda_{j_t}^{(k+1)}$ corresponds to applying an all-pass filter
 $\Phi_{j_t}(s) = (s-\lambda_{j_t}^{(k+1)})/(s+\overline{\lambda_{j_t}^{(k+1)}})$ that changes the phase of the approximant, see \cite{Hendrickx06someremarks}.
Let
 $\widehat{n}^{(k)}/\widehat{d}^{(k)}$ be the  barycentric representation corresponding to the nodes $\widehat{\lambda}_j^{(k+1)}$. Then, \textsf{VF} proceeds by solving the LS problem $\|\AC(\widehat{\bflambda}^{(k+1)}) \widehat{x}^{(k+1)} - h \|_2\longrightarrow\min$, instead of
 $\|\widehat{\Delta}^{(k)}(\AC(\widehat{\bflambda}^{(k+1)}) \widehat{x}^{(k+1)} - h )\|_2\longrightarrow\min$.
 This is not formally correct -- since the poles are changed by an external intervention, pole relocation does not compensate diagonal scaling.

 To make this step formally correct and interpretable in the framework of numerical linear algebra, we need  the barycentric representation
 $\widehat{n}^{(k)}(s)/\widehat{d}^{(k)}(s)$ of $\widetilde{n}^{(k)}(s)/\widetilde{d}^{(k)}(s)$, and the corresponding diagonal scaling $\widehat{\Delta}^{(k)}=\mathrm{diag}(1/|\widehat{d}^{(k)}(\xi_i)|)_{i=1}^\ell$ expressed using the new poles $\widehat{\lambda}_j^{(k+1)}$
(cf. the proof of Proposition \ref{Prop:Sobolev-VF}). Such a representation can be directly written down using
 \begin{equation}\label{eq:n/d-Ansatz}
 \frac{\widetilde{n}^{(k)}(s)}{\widetilde{d}^{(k)}(s)} \equiv
\frac{\widehat{n}^{(k)}(s)}{\widehat{d}^{(k)}(s)} = \frac{\sum_{j=1}^r\frac{\alpha_j^{(k)}}{s-\hat{\lambda}_j^{(k+1)}}}
{\sum_{j=1}^{{p}}\frac{\beta_j^{(k)}}{s-\hat{\lambda}_j^{(k+1)}} + 1},\;\; \widehat{d}^{(k)}(s) =
\sum_{j=1}^{{p}}\frac{\beta_j^{(k)}}{s-\hat{\lambda}_j^{(k+1)}} + 1,
\end{equation}
where the $\beta_j^{(k)}$'s must be determined so that the zeros of $\widehat{d}^{(k)}(s)$ are $\lambda_{j_t}^{(k+1)}$, $t=1,\ldots, p$.
This is an eigenvalue assignment problem in disguise and we use \cite{mehrmann1996app} to get
\begin{equation}\label{eq:beta-j-k}
\beta_j^{(k)} = \frac{{ \prod_{\ell=1}^p ( \hat{\lambda}_j^{(k+1)} + \hat{\lambda}_\ell^{(k+1)})}}
{{ \prod_{\ell=1,\ell\neq j}^p (\hat{\lambda}_j^{(k+1)} - \hat{\lambda}_\ell^{(k+1)})}},\;j=1,\ldots, p.
\end{equation}
\begin{proposition}\label{PROP:Justify-mirror-poles}
Let $\widehat{d}^{(k)}(s)$ be defined as in (\ref{eq:n/d-Ansatz}), (\ref{eq:beta-j-k}). Then for any $\omega \in\mathbb{R}$,
$|\widehat{d}^{(k)}(\imunit\omega)|=1$ and the diagonal scaling matrix $\widehat{\Delta}^{(k)}$  is unitary;
 the solution does not change from that of the unscaled problem.
\end{proposition}
\begin{proof}
Note that
$$
\prod_{j=1}^p (\imunit\omega - \hat{\lambda}_j^{(k+1)})\, \widehat{d}^{(k)}(\imunit\omega) =
\prod_{j=1}^p (\imunit\omega - \hat{\lambda}_j^{(k+1)}) (\sum_{j=1}^p\frac{\beta_j^{(k)}}{\imunit\omega-\hat{\lambda}_j^{(k+1)}} + 1)
= \prod_{j=1}^p (\imunit\omega + \hat{\lambda}_j^{(k+1)}).
$$
Recall that the $\hat{\lambda}_j^{(k+1)}$, $j=1,\ldots, p$, are closed under complex conjugation. The claim follows.
\end{proof}

Proposition \ref{PROP:Justify-mirror-poles} justifies proceeding with the same \textsf{VF} scheme after mirroring unstable poles,
as if nothing had happened. The same applies to the \textsf{SobVF} approximation described in \S \ref{SS=Sobolev-VF}.

\subsection{Numerical convergence and stopping criterion} \label{S=Convergence}
A theoretical convergence analysis of \textsf{VF} that determines conditions on $H(s)$ and the sampling nodes so as to guarantee convergence of \textsf{VF} remains an open problem.
An instructive  analysis by Lefteriu and Antoulas \cite{Lefteriu-Antoulas-2013} showed
(using a synthetic example with $r=2$) that the fixed points of the \textsf{VF} iterations can actually
be repellant and so  that the iteration may diverge.
Convergence behavior in realistic, large-scale settings appears
not yet to have been analyzed, and, to the best of our knowledge,
there are no published stopping criteria for the \textsf{VF} iteration that can be justified
rigorously by a rigorous  error or perturbation analysis.
In this section, we try to shed some light on these issues.

Assume now the setting of \S \ref{subsec:VF} with an ideal  convergence scenario:
Suppose that for some index $k$,
 the zeros and the poles of $\widetilde{d}^{(k)}(s)$ can be
\emph{numerically matched}, so that
$\lambda_{j}^{(k+1)} \approx \lambda_j^{(k)}$, and hence $\widetilde{d}^{(k)}(s) \cong 1$.
 Restated, this means that
the optimal matching distance
\begin{equation} \label{eq:omdistance}
\Omega_{k} = \min_{\sigma\in\mathbb{S}_r}\max_{j=1:r}|\lambda_j^{(k)} - \lambda_{\sigma(j)}^{(k+1)}|\;\;\;\;\;
\mbox{(here $\mathbb{S}_r$ denotes the permutation group)}
\end{equation}
between $(\lambda_j^{(k+1)})_{j=1}^r$ and $(\lambda_j^{(k)})_{j=1}^r$ as well as  $\max_j|\widetilde{\varphi}_j^{(k)}|$ are all
\emph{sufficiently small}.
The important tasks that arise here are determining $k$ and quantifying and justifying how small is
``\emph{sufficiently small}" ?
The following observations provide the key insights.

 \emph{(i)} Recall that $\bflambda^{(k+1)}$ is the spectrum of $\mathrm{diag}(\bflambda^{(k)}) + \widetilde{\bfvarphi}^{(k)}\mathbf{e}^T$, and thus can
be considered as the spectrum of a rank-one perturbation of the matrix $\mathrm{diag}(\bflambda^{(k)})$.
Hence, by \cite[Exercise VIII.3.2]{Bhatia-MatrixAnalysis-1997},
\begin{equation}\label{eq:opt_match_bound}
\Omega_k \leq (2r-1) \|\widetilde{\bfvarphi}^{(k)}\mathbf{e}^T\|_2 \leq \sqrt{r} (2r-1)\|(\widetilde{\varphi}_j^{(k)})_{j=1}^r\|_2
\leq r(2r-1) \max_j|\widetilde{\varphi}_j^{(k)}|,
\end{equation}
where $\Omega_k$ is  the optimal matching distance defined in
(\ref{eq:omdistance}).
In other words, by monitoring $\widetilde{\bfvarphi}^{(k)}$, we can determine in advance when
 $\lambda_j^{(k)}$ converges
(up to a predetermined tolerance) and thus end the pole identification phase.

\emph{(ii)} Moreover, it can be checked that, with proper  permutation matching used to enumerate $(\lambda_j^{(k+1)})_{j=1}^r$,
the element-wise relative differences between $\AC(\bflambda^{(k+1)})$ and $\AC(\bflambda^{(k)})$ are bounded by
\begin{equation} \label{eq:Aijdiff}
\max_{i,j} |(\AC(\bflambda^{(k)})_{ij} - \AC(\bflambda^{(k+1)})_{ij})/\AC(\bflambda^{(k+1)})_{ij} | \leq \frac{\Omega_k}{\mu_k},
\;\;\mbox{where}\;\;\mu_{k} = \min_{i=1:\ell}\min_{j=1:r} |\xi_i-\lambda_j^{(k)}| .
\end{equation}
Note that we can use (\ref{eq:opt_match_bound}) to estimate in advance that the  difference
(\ref{eq:Aijdiff})
 is less than
given $\epsilon$ by checking if $\Omega_k\leq \mu_k \epsilon$, i.e., if $\max_j|\widetilde{\varphi}_j^{(k)}| \leq  \mu_k \epsilon /(2r^2-r)$.

\emph{(iii)} Finally, another plausible and justifiable backward stable stopping criterion with a given tolerance threshold $\varepsilon$
can be seen in (\ref{tildephi}) with $k\leftarrow k-1$ as follows: From the estimate
$$
{\left| \sum_{j=1}^r \frac{\widetilde\varphi_j^{(k)}}{\xi_i-\lambda_j^{(k)}} \right|} \leq
\sqrt{r} \frac{\|(\widetilde\varphi_j^{(k)})_{j=1}^r\|_2}{\mu_k}\leq r \max_{j=1:r} |\widetilde\varphi_j^{(k)}|
\frac{1}{\mu_{k}},\;\; 
$$
valid for all $i=1,\ldots, \ell$,
where $\mu_k$ is as defined in (\ref{eq:Aijdiff}),
we conclude that if $\max_j|\widetilde{\varphi}_j^{(k)}| \leq \varepsilon \mu_{k}/r$, the
residue identification is simple because $\widetilde{n}^{(k)}(s) = \sum_{j=1}^r \frac{\widetilde\phi_j^{(k)}}{s-\lambda_j^{(k)}}$
can be taken as the final approximant in the pole-residue representation but now with a relative backward error of at most
$\varepsilon$ in the measurements $H(\xi_i)$.
However, to be on the safe side, the common practice of \textsf{VF} is to use  the ``converged" poles and then solve the LS problem $\|\AC^{(k)}(:,1:r) (\phi_j^{(k)})_{j=1}^r - h\|_2\rightarrow\min$ to determine the residues.

In practice, when the \textsf{VF} iterations converge, one observes that $\|(\widetilde\varphi_j^{(k)})_{j=1}^r\|_2$ tends to zero and  the estimate (\ref{eq:opt_match_bound}) reliably predicts the change in the nodes $(\lambda_j^{(k)})_{j=1}^r$ from step $k$ to step $k+1$.
However, if unstable nodes appear, they are mirrored as explained in \S \ref{SSS::Mirroring-and-scaling} and one works with the
$\widehat{\lambda}_j^{(k+1)}$'s instead of
the ${\lambda}_j^{(k+1)}$'s, which, in turn, means that (\ref{eq:opt_match_bound}) does not apply. In fact, it can happen that at each iteration until the very end, a subset of the
poles need to be flipped to $\Cplx_-$ and neither the $\widetilde{d}^{(k)}(s)$ converge to unity nor the nodes ${\lambda}_j^{(k)}$ settle as $k\rightarrow\infty$. That, however, does not necessarily means that the approximation is hopelessly bad. The following example illustrates this fact.
\begin{example}
{\em
We take the Beam model with $n=348$ from the \textsf{NICONET} collection  and obtain  order
$r=17$ and $r=18$ approximants using $\ell = 25$ conjugate pairs of logarithmically spaced nodes $\xi_i$.
The \textsf{VF} convergence history shown in Figure \ref{FIG:BEAM:convergence} illustrates two phenomena. In the  figure on the left with $r=18$, the value of
$\|(\widetilde{\varphi}_j^{(k)})_{j=1}^r\|_2$ settles around $8.042$ while the maximal relative change of the nodes drops down to the
level of $10^{-13}$. Thus,  \textsf{VF} converges but with $\|(\widetilde{\varphi}_j^{(k)})_{j=1}^r\|_2 \neq 0$.  The relative $\Hardy_2$ error norm of  the resulting approximatant is  $5.11\cdot 10^{-2}$.
The right figure, on the other hand, with $r=17$ shows a zigzag pattern for $\|(\widetilde{\varphi}_j^{(k)})_{j=1}^r\|_2$ (indicating two accumulation
points of the vectors $\widetilde{\bfvarphi}_j^{(k)}$, $k=1,2,\ldots$) and the $\mathcal{O}(1)$
relative changes in the nodes from step to step (indicating here too two accumulation points, where the zigzag comes from computing the relative distances). Thus, neither $\{\widetilde{\bfvarphi}^{(k)}\}$ nor
$\{\bfpol^{(k)}\}$ converges.
However, the iteration exhibits a periodicity in the behavior of the nodes.
With the lag of $2$ iterations, we see that $\| (\bfpol^{(k)} - \bfpol^{(k+2)})./\bfpol^{(k)} \|_\infty$ drops to the level of $10^{-10}$. In other words, the nodes cycle with the period of $2$.
The relative  $\Hardy_2$   error  norms of the approximants are around {$2.11\cdot 10^{-2}$ and  $2.45\cdot 10^{-2}$},
depending on the index $k$. It should be noted that the patterns shown on the right figure are not due to the flipping of unstable nodes. Even when that
mechanism is switched off, in this example we observe nearly the same periodic behaviors but in that case with an eventual unstable approximant, resulting in infinite
$\Hardy_2$ approximation error.
\begin{figure}
\begin{center}
 \includegraphics[width=2.95in,height=2.1in]{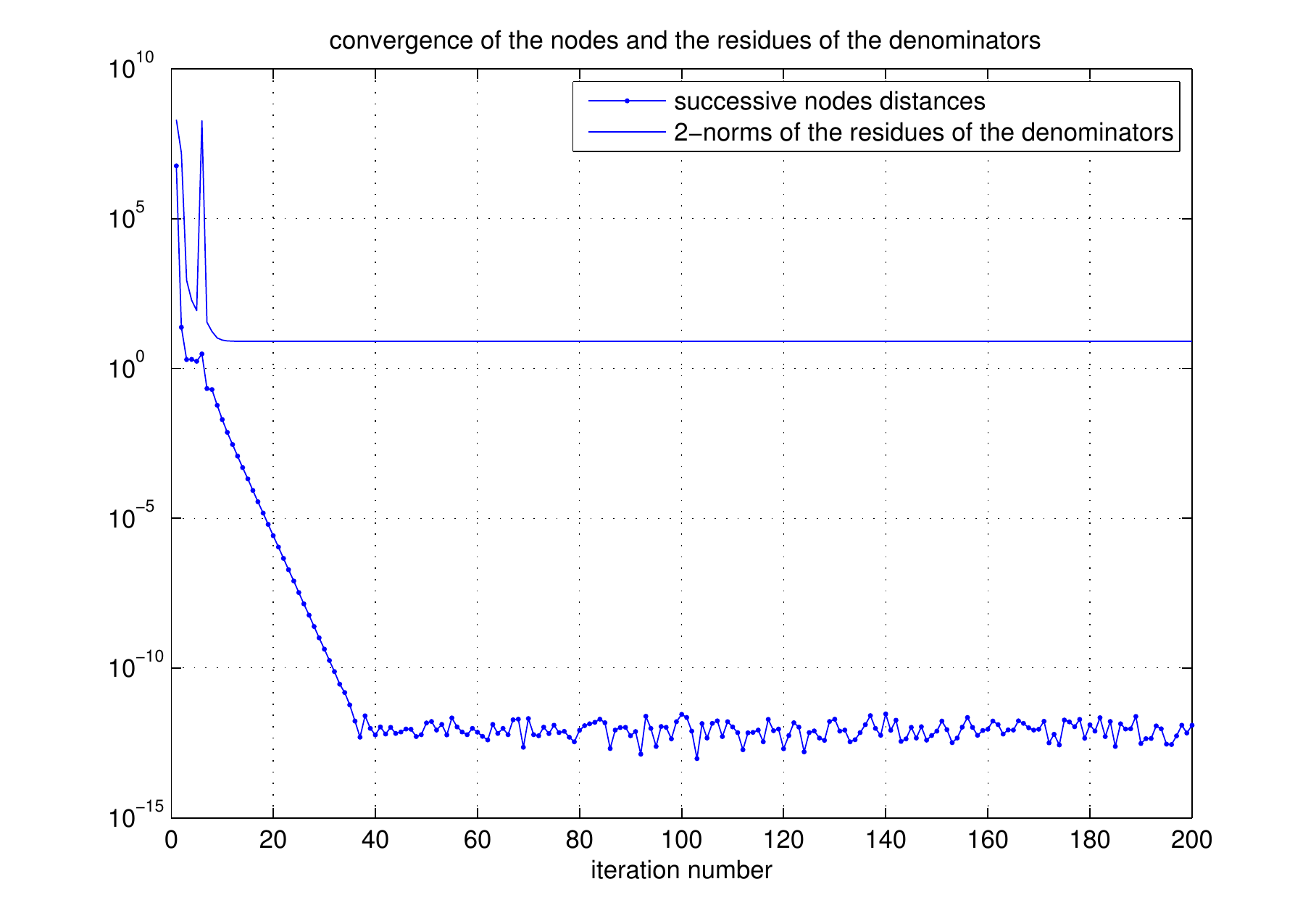}
 \includegraphics[width=2.95in,height=2.1in]{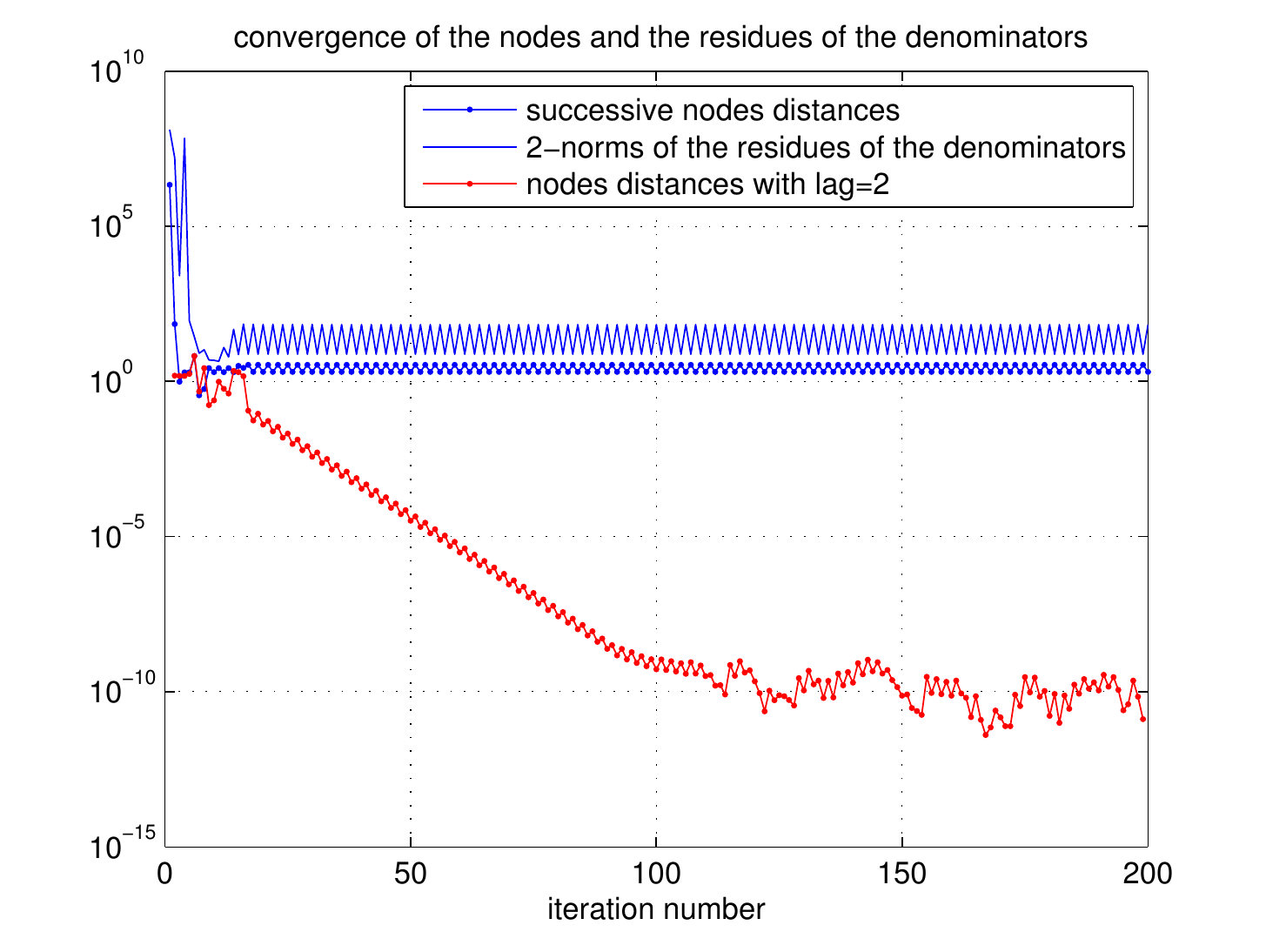}
 \end{center}
  \caption{\label{FIG:BEAM:convergence} Left ($r=18$): The distance between consecutive sets of computed poles,
  $\| (\bfpol^{(k)} - \bfpol^{(k+1)})./\bfpol^{(k)} \|_\infty$ and $\| (\widetilde{\varphi}_j^{(k)})_{j=1}^r\|_2$.
  Right ($r=17$): The relative differences $\| (\bfpol^{(k)} - \bfpol^{(k+1)})./\bfpol^{(k)} \|_\infty$  (jumping between $2.02$ and $3.41$),
  $\| (\bfpol^{(k)} - \bfpol^{(k+2)})./\bfpol^{(k)} \|_\infty$, and the residues of the denominators
  $\| (\widetilde{\varphi}_j^{(k)})_{j=1}^r\|_2$ (jumping between $7.51$ and $68.66$).
}
\end{figure}
}
\end{example}

\begin{remark}
{\em
The similar phenomenon is observed in \textsf{IRKA} as well, see \cite{Beattie-Drmac-Gugercin:2013-QIRKA}.
\emph{To cope with this behavior, the outer loop that governs the \textsf{VF} (or \textsf{IRKA}) iterations must have memory and be equipped with
a device capable of recognizing periodicity numerically (up to a tolerance).} Note that this is a more sophisticated
control of the iterations, where periodicity is just one of many possible events that can be captured.
For these types of iterations,
the usual memoryless loop breaking (comparing only consecutive steps, or testing against a stopping criterion) is not  enough. Instead,
for instance, a loop control with memory can be used for early detection of upcoming numerical convergence and better steering of
the iterations, see e.g. \cite{drm-ves-VW-2}.
Clearly, if the poles enter a periodic behavior, the distance $\delta_k = dist(\bfpol^{(k)},\bfpol^{(k-1)})$ will become periodic; and consequently
 it is enough to test the sequence $(\delta_k)$ for periodicity. If $\tau\geq 1$ is the estimated period, then
we have  $\tau$ candidate sets of poles $\bfpol^{(k)},\ldots, \bfpol^{(k+\tau-1)}$ for the approximation.
If these poles are not satisfactory, the looping must be interrupted.
Details are deferred to a subsequent work.
}
\end{remark}
\subsection{Avoiding ill-conditioning via regularization}\label{SS=CaveatConditioning}
The matrices $\BV$ in (\ref{eq:Bmatrix}) and $\AC$ in (\ref{eq:A=PC}) appearing in the \textsf{SK} and \textsf{VF} iterations, respectively, are composed of notoriously ill-conditioned
Vandermonde and Cauchy matrices. For instance, the spectral condition number $\kappa_2(V)=\|V\|_2 \|V^{-1}\|_2$
of an arbitrary real $n\times n$ Vandermonde matrix
 exceeds $2^{n-2}/\sqrt{n}$. For example, $\kappa_2(V)>10^{28}$ for $n=100$;
see, e.g., \cite{Tyrtyshnikov-94-HowBad,Beckermann-2000-CondNumbersVandermonde} for details.
Cauchy matrices can also be similarly as badly conditioned. The Hilbert matrix is the most famous one; for the
Hilbert matrix $\mathrm{Hilb}_{100}$ of order $100$, $\kappa_2(\mathrm{Hilb}_{100})>10^{150}$.
In addition to being  already ill-conditioned, the Cauchy matrices arising in \textsf{VF} appear  with additional scalings, as
$\Delta^{(k)}\AC^{(k)}=\Delta^{(k)} \begin{pmatrix} \Cauchy^{(k)} & D_\xi\Cauchy^{(k)} \end{pmatrix}$, where
$\Delta^{(k)}$ is as defined in (\ref{eq:Bmatrix}),
$\Cauchy^{(k)}$ is a  Cauchy matrix defined as $C_{ij}=\frac{1}{\xi_i-\lambda_j}$ for $i=1,\ldots,\ell$ and $j=1,\ldots,r$ and
$D_\xi$ is a diagonal matrix with  $(D_\xi)_{kk}= - H(\xi_k)$ for  $k=1,\ldots,\ell$.
 The diagonal
matrices $\Delta^{(k)}$ and $D_\xi$ can also be arbitrarily ill--conditioned. For instance, if  $H(s)$  has a pole in the vicinity
of $\xi_j$, then $|(D_\xi)_{jj}|=|H(\xi_j)|$ might be very large, especially much larger than $|H(\xi_i)|$ where $|\xi_i|$ is big and thus
$|H(\xi_i)|$ is small since $H(s)$ is assumed to be strictly proper.
{Hence,  the {LS} problem contains potentially
extremely ill-conditioned coefficient matrices, and
the normal equations approach, used in the early development of LS rational approximations \cite{Levy-1959,Sanathanan-Koerner-1963},
is in general not feasible.} One of the key improvements of  \textsf{VF}  \cite{Gustavsen-Semlyen-1999} is indeed
removing the scaling by $\Delta^{(k)}$ and using the unscaled matrix $\AC^{(k+1)}=\begin{pmatrix} \Cauchy^{(k+1)} & D_\xi\Cauchy^{(k+1)} \end{pmatrix}$
instead. However, this matrix still remains ill-conditioned.

Although \textsf{VF} and \textsf{SK} iteration perform effectively for smaller $r$ values and for not-too-pathological distributions of nodes and poles,
the high condition number of the underlying Cauchy and Vandermonde matrices has been recognized
as a serious obstacle for robust computations with higher order approximants on wider frequency ranges.
A discussion on how this ill-conditioning affects the quality of the approximation of the \textsf{SK} iterations, including illustrative examples,
is given in \cite{Soysal-Semlyen-1993}, where the authors demonstrate that equilibrating the columns of the {LS} coefficient matrix
in many cases dramatically improves the accuracy. Another similar preconditioning technique is frequency scaling
proposed in \cite{PintelonK05}.  In this section, we will propose a regularization-based approach to remedy ill-conditioning.

\subsubsection*{\textsf{RegVF}: Regularized Vector Fitting}
Increasing the order of the approximant naturally increases the potential of better approximation, but unfortunately only in theory. To illustrate this point, we continue the numerical experiment of Example \ref{Example:Sobolev-VF-Beam}, use the \emph{Beam} example and increase the order of the approximant from $r=40$ to
$r=80$. Recall that in Example \ref{Example:Sobolev-VF-Beam} with $r=40$, \textsf{VF} leads to a relative $\Hardy_2$ error of  $4.84\cdot 10^{-2}$.
However, when we increase $r$ to $r=80$, the relative error $\Hardy_2$ of \textsf{VF} increases to $1.31\cdot 10^2$. Of course, this apparent numerical divergence is solely due to the ill-conditioned {LS} problems, and in this context even the dimension $r=80$ can be considered large.

One possible cure is to regularize the solution. Towards this goal, we introduce the concept of \emph{Regularized Vector Fitting}, \textsf{RegVF}.
In \textsf{RegVF}, we  augment the {LS} coefficient matrix and replace the
original problem (\ref{eq:VF}) with
\begin{equation}\label{eq:VF:regularized}
\left\|  \begin{pmatrix} \AC^{(k+1)} \cr \begin{smallmatrix} \eta_1 I_r & 0 \cr 0 & \eta_2 I_r\end{smallmatrix}\end{pmatrix}\widetilde{x}^{(k+1)} - \begin{pmatrix} h \cr 0_{2 r\times 1}\end{pmatrix} \right\|_2\rightarrow\min,\;\; k = 0, 1, 2, \ldots,
\end{equation}
where $\eta_1$ and $\eta_2$ are the appropriately chosen regularization parameters. In a similar manner, we also regularize the final {LS} solution for the residue identification step. For the Beam model with $r=80$ and for the same nodes, this modification together with the choices of
$\eta_1=10^{-16}$, $\eta_2=\sqrt{\textsf{eps}}$ reduces  the relative error  from $1.31\cdot 10^2$
to $1.48\cdot 10^{-3}$. Needless to say, finding optimal regularization parameters in practice is far from trivial, because the
\emph{backslash} {LS} solver and the \textsf{svd()} function in Matlab are not a match for the highly ill-conditioned Cauchy-type matrices. For the sake of brevity, we omit the details to be included  in \cite{Drmac-Cauchy-SVD-2013}.

\subsubsection*{A note on row scaling}
To remedy ill-conditioning, in addition to column scaling, Soysal and Semlyen \cite{Soysal-Semlyen-1993} proposed other approaches such as frequency shifting and row scaling. We note that preconditioning by row scaling in the context of {LS} may not be allowed,
because it overrides carefully determined row weighting of a quadrature formula
(see \S \ref{SS=SamplingviaQuadrature}, \S \ref{SS=Sobolev-VF}), or row scaling designed to cope with measurement noise, see \S \ref{SS=noisy}.
The ill-conditioning induced by row-weighting can be partially overcome if the QR factorization is
computed with the full pivoting introduced by Powell and Reid \cite{Powell-Reid-1968} and analyzed by Cox and Higham \cite{cox-hig-98}.
Therefore, if row scaling is an issue, before using the \emph{backslash} {LS} solver in \textsf{VF},  we propose the following equally good yet more efficient simplified variant  of Powell-Reid pivoting,  due to {\AA}ke Bj\"{o}rck:

\begin{verbatim}

function x = LS_solve( A, b )
m = size(A,1) ; D = zeros(m,1) ;  for i = 1 : m, D(i) = norm(A(i,:),inf) ;  end
[~,P] = sort(D,'descend') ; A = A(P,:) ; b = b(P) ; x = A \ b ;

\end{verbatim}

\begin{remark}
{\em
It is well known that using orthonormal basis functions improves numerical stability of
approximation methods. For rational approximation schemes such as  \textsf{VF}, several authors have
developed methods based on orthogonal rational functions, e.g., \cite{Akcay-Ninness-1999,Deschrijver-Haegeman-Dhaene-2007}.
Examples where the orthonormal vector fitting (\textsf{OrthVF}) can outperform  \textsf{VF} are given in \cite{Antonini-Deschrijver-Dhaene-2006}.
However, this is still an open debate as Gustavsen \cite{Gustavsen-EMC-2007} points out that careful
implementation of \textsf{VF} with suitably chosen initial poles matches the performances of  \textsf{OrthVF}  on the same
examples used in \cite{Antonini-Deschrijver-Dhaene-2006}.
}
\end{remark}

\section{Vector Fitting using noisy data}\label{SS=noisy}
The starting point for the rational approximation framework we consider in this paper is a set of transfer function measurements/evaluations. Even though so far we have only considered noise-free data and even though a complete analysis of the underlying framework for \textsf{VF} in the presence of noise is not the main focus of this paper, in this section we provide a new formulation for \textsf{VF} for noisy data and illustrate that the pole reallocation feature of \textsf{VF} leads to a powerful mechanism for removing noise asymptotically as the iteration advances. We also propose a new numerical linear algebra framework for the noisy data case, and pose some challenging problems for future research.

\subsection{A mixed total least squares framework}
Suppose that in (\ref{GenVFprob}), instead of exact values $h_i = H(\xi_i)$, we have
 noisy measured  data:  $\widetilde{h}_i = h_i + \delta h_i$.  Assuming that measurement errors are
uncorrelated,
the proper formulation of the new LS problem is
\begin{equation}\label{eq:H=n/d:noisy}
\mbox{Find}\;\; H_r(s)= \frac{n(s)}{d(s)} \equiv  \frac{\sum_{j=0}^{r-1}\alpha_js^j }{ 1+ \sum_{j=1}^r \beta_j s^j}\;\;\mbox{such that}\;\;
\sum_{i=1}^\ell w_i^2 \left| \frac{n(\xi_i)}{d(\xi_i)} - \widetilde{h}_i \right|^2 \longrightarrow \min ,
\end{equation}
where the weight $w_i$ is the reciprocal of the standard deviation for the $i$th measurement - information that can be considered
as part of the measurement and is essential in guiding the approximation process. Neglecting the weights $w_i$ corresponds
to assuming the same variance across all measurements. Such an assumption is generally not realistic;  particularly when the measurements, $\widetilde{h}_i$, span a large range of values. This, in turn, will degrade the performance of \textsf{VF}, causing it hopelessly to try to
fit the noise.

Statistical properties of the errors, $\delta h_i$, can generally be obtained through repeated measurements, e.g. with periodic excitation,
and depending on the model (see e.g. \cite[\S IV.]{Pintelon-at-al-1994}), various formulations are obtained.
 In general, if we set 
 $$
 \bfeps=[~\eps_1,\eps_2,\ldots,\eps_\ell~]^T \quad \mbox{with}\quad \eps_i=\frac{n(\xi_i)}{d(\xi_i)}-\widetilde{h}_i \quad\mbox{for}~~ i=1,\ldots,\ell, 
   $$
  the problem  (\ref{eq:H=n/d:noisy})
can be re-formulated as $\|W \bfeps\|_2\rightarrow\min$, where $W$ is the inverse Cholesky factor of a positive definite variance-covariance matrix.
We assume for simplicity that errors are uncorrelated so that $W$ is diagonal,  and that further the individual error variances can be estimated reliably. 
Here we focus on the numerical linear algebra aspects of the problem. For  details on stochastic estimation of transfer functions, see, e.g.,
\cite{Kollar-1993,Pintelon-at-al-1994,DeVries-VanDerHof-1995,Pintelon-Schoukens-Rolain-13IFAC}.

The \textsf{SK} iteration, which forms the basis for \textsf{VF}, now takes the weighted form
$$
\left\| W \mathrm{diag}(\frac{1}{|d^{(k)}(\xi_i)|}) \left( \begin{smallmatrix} n^{(k+1)}(\xi_1) - d^{(k+1)}(\xi_1) \widetilde{h}_1\cr
\vdots \cr n^{(k+1)}(\xi_\ell) - d^{(k+1)}(\xi_\ell) \widetilde{h}_{\ell}\end{smallmatrix}\right)\right\|_2 \!\! \equiv\!
 \| W D^{(k)} (\widetilde{\AC}^{(k)} x^{(k+1)} - \widetilde{h})\|_2 \longrightarrow \min,  \;k=0,1,\ldots.
$$
In \textsf{VF}, due to pole relocation, the scaling factors $1/|d^{(k)}(\xi_i)|$ are dropped and $D^{(k)}\equiv I$. 
 The LS objective is $\|W (\widetilde{\AC}^{(k+1)} \widetilde{x}^{(k+1)} - \widetilde{h})\|_2\rightarrow\min$. 
 To ease the growing notational burden,
we drop the iteration index $k$ and set
$\AC \equiv \widetilde{\AC}^{(k+1)}$ and $x\equiv \widetilde{x}^{(k+1)}$. (If needed, we may assume that $k$ is big enough, so that the \textsf{VF} iterations
have reached numerical convergence.)
Compare the original LS problem in (\ref{eq:H=n/d:noisy}) with this linearized version.
Note that, in the process of linearization,
 noise that had appeared only in the right-hand side of (\ref{eq:H=n/d:noisy}) now enters the coefficient matrix,
leading to the minimization problem:  $\| W ( \AC x- (h+\delta h) )\|_2\rightarrow\min$,
where $\AC = \AC_{noise\rule{0.1cm}{0.4pt}free} + \delta \AC$. 

Since the tacit assumption of LS approximation is that only the right-hand side is contaminated with noise,
 a Total Least-Squares (TLS) formulation \cite{Golub-VanLoan-1980} appears to be more appropriate to this setting than the more typical LS formulation.
More precisely, allowing for noisy data in (\ref{eq:LS}), (\ref{eq:VF}), (\ref{eq:B=Vandermonde})
will lead to mixed LS/TLS problems.
Notice that from the definition (\ref{eq:A=PC}) only the last $r$ columns of the matrix
$\AC$ can be contaminated by noise, since $\xi_i$ and $\pol_j$ are considered exact.   
Thus, the perturbation $\delta \AC$ due to noise is
structured and closely related to the perturbation $\delta h$:
\begin{equation}\label{eq:TLS:structured_dA}
\delta \AC  = \begin{pmatrix} 0 & F \end{pmatrix},\;\;
F = \left(\frac{\delta h_i}{\xi_i-\lambda_j}\right)_{i,j=1,1}^{\ell,r} \equiv (\delta h \mathbf{e}^T)\circ C,\;\;
C_{ij}=\frac{1}{\xi_i-\lambda_j},\;\;\mathbf{e}^T = (1\; 1\; \ldots 1),
\end{equation}
where $\circ$ again denotes the Hadamard product. Note that $F$ has rank-one displacement
structure, i.e., it satisfies a Sylvester equation with a rank-one nonhomogeneity: 
$$\Xi F - F \Lambda = \delta h \,\mathbf{e}^T,\quad\mbox{ with }
\Xi = \mathrm{diag}(\xi_i)_{i=1}^\ell\ \mbox{ and }\ \Lambda = \mathrm{diag}(\pol_j)_{j=1}^r.
$$

We first consider how  \textsf{VF} fits into this general TLS framework and show that pole-relocation that is intrinsic to \textsf{VF} has useful additional consequences in this setting.
Recall that the minimization of $\| W(\AC x-\widetilde{h})\|_2$ can be
equivalently formulated as
\begin{equation} \label{eq:ls_equiv}
\| W \mathbf{r}\|_2\rightarrow\min, \;\;\mbox{subject to}\;\; \widetilde{h}+\mathbf{r}\in\mathrm{Range}(\AC).
\end{equation}
One may find the solution to (\ref{eq:ls_equiv})  by seeking the minimal change,
$\widetilde{h} \mapsto \widetilde{h} + \mathbf{r}$, 
(as measured in a $W$-weighted norm)
such that $\AC x = \widetilde{h} + \mathbf{r}$.
In the general TLS framework, 
a minimal change $\begin{pmatrix} \Delta \AC &  \mathbf{\widehat{r}}\end{pmatrix}$ is determined
(as measured now by weighted matrix norm: $\| W \begin{pmatrix} \Delta \AC &  \mathbf{r}\end{pmatrix} T\|_F$)
such that $(\AC+\Delta \AC) \widehat{x} = \widetilde{h}+\mathbf{\widehat{r}}$.
If the entries of $\begin{pmatrix} \Delta \AC &  \mathbf{r}\end{pmatrix}$ are uncorrelated, then $W$ and $T=\mathrm{diag}(t_i)_{i=1}^{2r+1}$ are diagonal matrices. For a detailed and instructive discussion on scaling, see \cite[\S 3.6.2]{VanHuffel-TLS-1991}.

When there is no structural requirement on the perturbation $\Delta \AC$, the TLS solution is computed as follows \cite{Golub-VanLoan-1980}:
Let $ G\equiv W \begin{pmatrix}  \AC &  \widetilde{h}\end{pmatrix} T = U \Sigma V^*$ be the SVD, and assume for simplicity that the smallest
singular value $\sigma_{2r+1}>0$ is simple, with the corresponding singular vectors $u_{2r+1}$ (left) and $v_{2r+1}$ (right). Further assume that the last
component of $v_{2r+1}$ is nonzero; i.e.
$v_{2r+1} = \left(\begin{smallmatrix}z \cr \eta \end{smallmatrix}\right)
$
where $z \in \Cplx^{2r}$, $\eta \in \Cplx$ with $\eta \neq 0$ .
Then,  the minimal perturbation $\begin{pmatrix} \Delta \AC &  \mathbf{\widehat{r}}\end{pmatrix}$ and the corresponding
solution $\widehat{x}$ are given explicitly as
\begin{equation}\label{eq:TLS-solution}
\begin{pmatrix} \Delta \AC \! & \!  \mathbf{\widehat{r}}\end{pmatrix} = - \sigma_{2r+1}W^{-1} u_{2r+1}v_{2r+1}^* T^{-1}, \;\;
\widehat{x}= \frac{-1}{\eta \,t_{n+1}}\mathrm{diag}(t_i)_{i=1}^{2r}\, z.
\end{equation}
For more details on the solution procedure see \cite[Algorithm 3.1]{VanHuffel-TLS-1991}.

Considering the special structure (\ref{eq:TLS:structured_dA}) of the perturbation in our setting,
we  formulate the following structured mixed LS/TLS problem: With $W$ as before and $T=\mathrm{diag}(t_i)_{i=1}^{r+1}$, solve
\begin{equation}\label{eq:TLS}
\| W \begin{pmatrix} E & \mathbf{r} \end{pmatrix} T \|_F\longrightarrow \min, \mbox{~subject to}\;\; h+\mathbf{r} \in\mathrm{Range}(\AC+\begin{pmatrix} 0 & E\end{pmatrix})
\;\mbox{~and~}\; \Xi E - E \Lambda = \mathbf{r} e^T .
\end{equation}
Since $W \begin{pmatrix} E & \mathbf{r} \end{pmatrix} T = (T\otimes W) \left(\begin{smallmatrix} \mathrm{vec}(E) \cr \mathbf{r} \end{smallmatrix}\right)$,
the objective function in (\ref{eq:TLS}) can be re-written as
\begin{equation}\label{eq:TLS(r)}
\|W \begin{pmatrix} E & \mathbf{r} \end{pmatrix} T \|_F^2
= \left\| \left(\begin{smallmatrix} t_1 W &  &  \cr
 & \ddots &  \cr & & t_{r+1} W\end{smallmatrix}\right) \left( \begin{smallmatrix} \mathbf{r}\circ C(:,1) \cr \vdots \cr
  \mathbf{r}\circ C(:,r) \cr \mathbf{r} \end{smallmatrix}\right)\right\|_F^2
  = t_{r+1}^2 \|W\mathbf{r}\|_2^2 + \sum_{j=1}^r t_j^2 \| W (\mathbf{r}\circ C(:,j))\|_2^2.
\end{equation}
Depending on $T$ and the distribution of the $\xi_i$'s and the $\lambda_j$'s, minimizing the above expression is
related to minimizing $\|W\mathbf{r}\|_2$.
For example consider the case of the structured perturbations $E=\mathbf{r}e^T$ (not of the type we have here\footnote{Take very low frequencies and all
$\lambda_j$'s around $-1$, or consider frequency scaling to approximate the desired structure.}, but instructive to consider) and $T=I$ (reasonable in this situation). In this case,
the minimization problem $\| W \begin{pmatrix} E & \mathbf{r} \end{pmatrix} T \|_F\longrightarrow \min$  is indeed equivalent to $\|W \mathbf{r}\|_2 \longrightarrow \min$. But in general, developing a theory for solvability and a robust numerical algorithm for solving
(\ref{eq:TLS}) is  a challenging problem.
If we assume to have found the minimizing $E$ and $\mathbf{\widehat{r}}$, the solution $\widehat{x}$ is, then, defined by
$(\AC+\begin{pmatrix} 0 & E\end{pmatrix})\widehat{x} = h + \mathbf{\widehat{r}}$. Otherwise (e.g., if (\ref{eq:TLS}) has no solution), ignore the rank-one displacement structure,
and use the solution of the mixed LS/TLS problem, computed using \cite[Algorithm 3.2]{VanHuffel-TLS-1991}, or the solution (\ref{eq:TLS-solution}) of
the TLS problem. With these two cases, we obtain two new variants of \textsf{VF}, denoted by \textsf{LS/TLS-VF} and \textsf{TLS-VF}, respectively.
As stated above, the special structure of the coefficients matrices in  \textsf{LS/TLS-VF} makes it a  challenging problem. Assuming the existence of a solution to these structured problems, 
numerically sound implementations to obtain the solution will depend on developing accurate numerical linear algebra
tools, e.g., accurate SVD computations, for  Cauchy-type matrices that arise in \textsf{VF}.
\subsection{\textsf{VF} as an asymptotic LS/TLS procedure}
We compare the LS solution in  step $k$ of \textsf{VF} to the solution of $( \AC + \begin{pmatrix} 0 & E \end{pmatrix} ) \hat{x} = h + \mathbf{\hat{r}}$ in  step $k$ of \textsf{LS/TLS-VF}.
Recall that the LS problem minimizes $\|W\mathbf{r}\|_2$, and the solution $x$ satisfies $\AC x = h+\mathbf{r}$.
If we partition $x$ as $x=\left(\begin{smallmatrix} \bfphi \cr \bfvarphi\end{smallmatrix}\right)$, and add the LS/TLS error in $\AC$,
we obtain
$$
( \AC + \begin{pmatrix} 0 & E \end{pmatrix} ) x = h + \mathbf{r} + E\bfvarphi,\;\;
$$ 
where  in the case of (\ref{eq:TLS}), $E\bfvarphi = ((\mathbf{\widehat{r}} e^T)\circ C)\bfvarphi = \mathbf{\widehat{r}}\circ(C\bfvarphi)).
$
Since  the $\bfvarphi$-part of $x$ in \textsf{VF} converges to zero,
it holds that $\|E\bfvarphi\|_2 \leq \|\mathbf{\widehat{r}}\|_2 \|C\bfvarphi\|_2$ is small relative to $\|\mathbf{\widehat{r}}\|_2$.
Recall that  minimizing $\|W\mathbf{r}\|_2$ and $\|W\begin{pmatrix} E & \mathbf{r}\end{pmatrix}\|_2$ from (\ref{eq:TLS(r)})
are related.
This  reveals another silent, yet powerful,
feature of \textsf{VF} that makes it much more than a reformulation of SK iteration. Asymptotically, thanks to the persistent change of representation through pole relocation, \textsf{VF} is (approximately) performing structured mixed LS/TLS minimization.

\subsection{A diagonally-restricted LS/TLS formulation}
In the previous section, we discussed the TLS approach to \textsf{VF} in the presence of noise and  by comparing with a generic LS/TLS procedure we showed
that the original formulation of \textsf{VF} will approximately solve the LS/TLS problem.  In this section, 
we will introduce a new framework, that we believe is the
correct formulation to perform \textsf{VF} in the presence of noise. 

It follows from (\ref{eq:ls_equiv}) and the definition of $C$ in (\ref{eq:TLS:structured_dA}) that the objective function $\|W \mathbf{r} \|_2$ 
is minimized with respect to
the condition
\begin{displaymath}
( \begin{pmatrix} C & \mathrm{diag}(\widetilde{h}) C \end{pmatrix} + \begin{pmatrix} 0 & \mathrm{diag}(\mathbf{r}) C\end{pmatrix} ) x = \widetilde{h} + \mathbf{r},
\;\mbox{i.e.}\; (\begin{pmatrix} C & \mathrm{diag}(\widetilde{h}) C & \widetilde{h} \end{pmatrix}
+ \mathrm{diag}(\mathbf{r}) \begin{pmatrix} 0 &  C & \mathbf{e} \end{pmatrix})\left(\begin{smallmatrix} x \cr -1\end{smallmatrix}\right) = 0 .
\end{displaymath}
Define $Z = \begin{pmatrix} C & \mathrm{diag}(\widetilde{h})C & \widetilde{h} \end{pmatrix}$, $S=\begin{pmatrix} 0 & C & \mathbf{e}\end{pmatrix}$.
Then we propose a diagonally-restricted LS/TLS formulation in step $k$ of \textsf{VF}, stated as follows:   Find  $x=\left(\begin{smallmatrix} \bfphi \cr \bfvarphi\end{smallmatrix}\right)$ as the solution (if it exists) of the
constrained minimization problem
\begin{equation}\label{eq:challenge}
\min_{\mathbf{r}} \{ \|W \mathbf{r}\|_2 \; :\; (Z + \mathrm{diag}(\mathbf{r}) S) \left( \begin{smallmatrix}  x \cr -1 \end{smallmatrix}\right) =0\}.
\end{equation}
Set $\hat{Z}=WZ$ and note that $\hat{R} \equiv W \mathrm{diag}(\mathbf{r})$ is the minimal perturbation $\hat{Z} \leadsto \hat{Z} + \hat{R}S$ that makes $\hat{Z}$ singular. Apart from the special structure of $\hat{R}$, this is related to the notion of  restricted singular values \cite{Zha:1991:RSV:104085.104100} of the
matrix triplet $(\hat{Z}, I, S)$:
$
\sigma_k (Z,I,S) = \min \{ \|\Theta\|_2 \; :\; \mathrm{rank}(Z + I \Theta S) \leq k-1\}.
$
This connection, explained in \cite{VanHuffel:1991:RTL:105724.105732}, together with  methods presented in \cite{Beck:2007:MTL:1265603.1265657}    form the starting point for attacking the problem of  solving  (\ref{eq:challenge}) numerically. These issues will be explored in future work.

\section{Conclusions and Future Directions}
\textsf{VF} has been widely and successfully used.
Notwithstanding substantial advances and many successful applications of the method,
analytical justification of its success from numerical linear algebra and rational approximation perspectives has been missing.
This work is a step toward filling that gap. Noting first that a small \textsf{VF} fitting error 
does not necessarily correspond to small approximation error, we
 related \textsf{VF} to discrete $\Hardy_2$ minimization and proposed a quadrature-based version, called \textsf{QuadVF}, which improves performance dramatically. We extended \textsf{VF} to include a derivative penalty in the LS minimization by performing a quadrature-based discretization of a continuous Sobolev norm, leading to a method we called \textsf{SobVF}.  We also analyzed several practical and numerical issues arising in \textsf{VF} using a rigorous theoretical framework. For example, we  analytically justified the mechanism behind the mirroring of unstable poles during  \textsf{VF}. We investigated the numerical convergence of \textsf{VF} and illustrated different scenarios for divergence that could arise.  One of the major numerical issues that can arise in \textsf{VF} is the appearance of highly ill-conditioned coefficient matrices; we offered a remedy via regularization. Even though most of our analyses assume exact data, we briefly considered \textsf{VF} in the case of noisy data and
showed the utility of a mixed LS/TLS framework.

Aside from the newly developed, effective methods that are described here,  our work also
leads to a variety of challenging theoretical and practical issues that will be explored in 
subsequent work.  These include: effective regularization techniques, refined computational
strategies for the diagonally-restricted LS/TLS formulation of \textsf{VF} introduced in (\ref{eq:challenge}), extensions to the multiple-input/multiple-output case via tangential interpolation (reflecting the structure of the underlying $\Hardy_2$ setting), and adaptive determination of apppropriate reduced dimension (say, informed by the Loewner framework developed in \cite{Mayo-2007,AIL11,IonitaAntoulas2013,lefteriu2010new}).

\bibliography{ModRed}
\bibliographystyle{siam}

\end{document}